\long\def\commentout#1{}

\newif\ifprint
\printfalse
\documentclass[twoside,10pt]{HandbookOfModuli}

\usepackage{amsmath,amsthm}
\usepackage{amssymb}
\usepackage{latexsym}
\usepackage[utf8]{inputenc}
\usepackage{textcomp}
\usepackage{color}
\usepackage{titlesec}
\usepackage{xspace}

\ifprint
	\input{giovanni} 
	\usepackage[final]{microtype}
\fi
\usepackage{bm}
\renewcommand{\mathbf}[1]{\bm{#1}} 

\ifprint
	\definecolor{linkred}{rgb}{0,0,0} 
	\definecolor{linkblue}{rgb}{0,0,0} 
\else
	\definecolor{linkred}{rgb}{0.7,0.2,0.2}
	\definecolor{linkblue}{rgb}{0,0.2,0.6}
\fi

\numberwithin{equation}{section} 

\usepackage[
	hypertexnames=false,
	hyperindex,
	pagebackref,
	pdftex,
	pdftitle={Handbook of Moduli},
	pdfdisplaydoctitle,
	pdfpagemode=UseNone,
	breaklinks=true,
	extension=pdf,
	bookmarks=false,
	plainpages=false,
	colorlinks,
	linkcolor=linkblue,
	citecolor=linkred,
	urlcolor=linkred,
	pdfmenubar=true,
	pdftoolbar=true,
	pdfpagelabels,
	pdfpagelayout=TwoPageRight,
	pdfview=Fit,
	pdfstartview=Fit
]{hyperref}

\catcode`@=11 \baselineskip 4.5mm
\def\ps@handbook{\def\@oddhead{\hfill \leftmark \hfill\thepage }
\def\@evenhead{\thepage \hfill \rightmark \hfill}
\def\@oddfoot{}
\def\@evenfoot{}}
\def\@evenhead{}
\def\@oddfoot{}
\def\@evenfoot{\hfill\copyright\ China Higher Education Press}
\def\list#1#2{\ifnum \@listdepth >5\relax \@toodeep \else \global
\advance \@listdepth\@ne \fi \rightmargin \z@ \listparindent\z@
\itemindent\z@ \csname @list\romannumeral\the\@listdepth\endcsname
\def\@itemlabel{#1}\let\makelabel\@mklab \@nmbrlistfalse #2\relax
\@trivlist \parskip -\parsep \parindent\listparindent \advance
\linewidth -\rightmargin \advance\linewidth -\leftmargin \advance
\@totalleftmargin \leftmargin \parshape \@ne \@totalleftmargin
\linewidth \ignorespaces}
\renewcommand*\l@section{\@tocline{1}{0pt}{0em}{1.75em}{}}
\renewcommand*\l@subsection{\@tocline{2}{0pt}{1.75em}{2em}{}} 
\renewcommand\contentsnamefont{\bfseries\large} 
\catcode`@=12
\renewcommand{\theequation}{\thesection.\arabic{equation}}

\def\thebibliography#1{\section*{References}
\list{[\arabic{enumi}]}{\settowidth \labelwidth{[#1]} \leftmargin
\labelwidth \advance \leftmargin \labelsep \usecounter{enumi}}
\def\newblock{\hskip .11em plus .33em minus .07em} \sloppy
\clubpenalty 4000 \widowpenalty 4000 \sfcode`\.=1000 \relax}

\titleformat{\section}{\normalfont\large\bfseries}{\thesection.}{0.5em}{}[\kern0.em]
\titleformat{\subsection}{\normalfont\bfseries}{\thesubsection.}{0.3em}{}[\kern0.em]
\titleformat{\subsubsection}[runin]{\normalfont\bfseries}{\thesubsubsection.}{0.5em}{}[\kern0.5em]

\def\fofsubsubsection#1{\refstepcounter{equation}\subsubsection*{\theequation.\kern0.25em #1}}
\def\foisubsubsection#1{\refstepcounter{equation}\subsubsection*{\kern\parindent\theequation.\kern0.25em #1}}

\usepackage{amscd}

\theoremstyle{plain} 

\newtheorem{theorem}{\sc Theorem}[section]

\newtheorem{lemma}[theorem]{\sc Lemma}

\newtheorem{proposition}[theorem]{\sc Proposition}

\theoremstyle{definition}

\newtheorem{definition}[theorem]{\sc Definition}

\newtheorem{example}[theorem]{\sc Example}

\newtheorem{examples}[theorem]{\sc Examples}

\newtheorem{remarks}[theorem]{\sc Remarks}

\newcommand\bA{{\mathbb A}}

\newcommand\bG{{\mathbb G}}

\newcommand\bN{{\mathbb N}}
\newcommand\bP{{\mathbb P}}
\newcommand\bQ{{\mathbb Q}}
\newcommand\bR{{\mathbb R}}

\newcommand\bZ{{\mathbb Z}}

\newcommand\cD{{\mathcal D}}

\newcommand\cF{{\mathcal F}}
\newcommand\cG{{\mathcal G}}
\newcommand\cH{{\mathcal H}}
\newcommand\cI{{\mathcal I}}
\newcommand\cJ{{\mathcal J}}

\newcommand\cM{{\mathcal M}}
\newcommand\cN{{\mathcal N}}
\newcommand\cO{{\mathcal O}}

\newcommand\cR{{\mathcal R}}

\newcommand\cV{{\mathcal V}}
\newcommand\cW{{\mathcal W}}
\newcommand\cX{{\mathcal X}}
\newcommand\cZ{{\mathcal Z}}

\newcommand\fg{{\mathfrak g}}
\newcommand\fh{{\mathfrak h}}
\newcommand\ft{{\mathfrak t}}

\newcommand\tg{\tilde{g}}
\renewcommand\th{\tilde{h}}
\newcommand\tlambda{\tilde{\lambda}}
\newcommand\tp{\tilde{p}}
\newcommand\tq{\tilde{q}}

\newcommand\tG{\tilde{G}}
\newcommand\tM{\tilde{M}}
\newcommand\tV{\tilde{V}}
\newcommand\tX{\tilde{X}}
\newcommand\tY{\tilde{Y}}
\newcommand\tZ{\tilde{Z}}

\newcommand\tcZ{\tilde{\mathcal Z}}

\newcommand\ualpha{\underline{\alpha}}
\newcommand\ulambda{\underline{\lambda}}

\newcommand\ad{{\rm ad}}

\newcommand\gr{{\rm gr}}

\newcommand\id{{\rm id}}

\renewcommand\log{{\rm log}}

\newcommand\reg{{\rm reg}}
\newcommand\rk{{\rm rk}}

\newcommand\Aut{{\rm Aut}}

\newcommand\GL{{\rm GL}}
\newcommand\Gr{{\rm Gr}}
\newcommand\Hilb{{\rm Hilb}}
\newcommand\Hom{{\rm Hom}}
\newcommand\Ind{{\rm Ind}}
\newcommand\Irr{{\rm Irr}}

\newcommand\M{{\rm M}}
\newcommand\Mor{{\rm Mor}}
\newcommand\Pic{{\rm Pic}}
\newcommand\Proj{{\rm Proj}}
\newcommand\PGL{{\rm PGL}}
\newcommand\PSL{{\rm PSL}}
\newcommand\SL{{\rm SL}}
\newcommand\Spec{{\rm Spec}}

\newcommand\Sym{{\rm Sym}}
\newcommand\Univ{{\rm Univ}}

\title{Invariant Hilbert schemes}

\author{Michel Brion}

\date{}

\begin{document}

\begin{abstract}
This is a survey article on moduli of affine schemes equipped with an action of a reductive group.
The emphasis is on examples and applications to the classification of spherical varieties.
\end{abstract}

\maketitle

\tableofcontents

\section{Introduction}
\label{sec:introduction}

The Hilbert scheme is a fundamental object of projective algebraic geometry; it parametrizes
those closed subschemes of the projective space $\bP^N$ over a field $k$, that have a prescribed 
Hilbert polynomial. Many moduli schemes (for example, the moduli space of curves of a fixed genus)
are derived from the Hilbert scheme by taking locally closed subschemes and geometric invariant theory 
quotients. 

In recent years, several versions of the Hilbert scheme have been constructed in the setting of 
algebraic group actions on affine varieties. One of them, the \emph{$G$-Hilbert scheme} $G$-$\Hilb(V)$, 
is associated to a linear action of a finite group $G$ on a finite-dimensional complex vector space $V$; 
it classifies those $G$-stable subschemes $Z \subset V$ such that the representation of $G$ in the coordinate 
ring $\cO(Z)$ is the regular representation. The $G$-Hilbert scheme comes with a morphism to the quotient
variety $V/G$, that associates with $Z$ the point $Z/G$. This \emph{Hilbert-Chow morphism} has an inverse over 
the open subset of $V/G$ consisting of orbits with trivial isotropy group, as every such orbit $Z$ is 
a point of $G$-$\Hilb(V)$. In favorable cases (e.g. in dimension $2$), the Hilbert-Chow morphism is 
a desingularization of $V/G$; this construction plays an essential role in the McKay correspondence
(see e.g. \cite{IN96, IN99, BKR01}).

Another avatar of the Hilbert scheme is the \emph{multigraded Hilbert scheme} introduced by Haiman and
Sturmfels in \cite{HS04}; it parametrizes those homogeneous ideals $I$ of a polynomial ring 
$k[t_1,\ldots,t_N]$, graded by an abelian group $\Gamma$, such that each homogeneous component of 
the quotient $k[t_1,\ldots,t_N]/I$ has a prescribed (finite) dimension. In contrast to the construction
(due to Grothendieck) of the Hilbert scheme which relies on homological methods, that of Haiman and
Sturmfels is based on commutative algebra and algebraic combinatorics only; it is valid over any base 
ring $k$. Examples of multigraded Hilbert schemes include the Grothendieck-Hilbert scheme (as follows 
from a result of Gotzmann, see \cite{Go78}) and the \emph{toric Hilbert scheme} (defined by Peeva and Sturmfels
in \cite{PS02}) where the homogeneous components of $I$ have codimension $0$ or $1$.

The invariant Hilbert scheme may be viewed as a common generalization of $G$-Hilbert schemes and 
multigraded Hilbert schemes; given a complex reductive group $G$ and a finite-dimensional $G$-module $V$, 
it classifies those closed $G$-subschemes $Z \subset V$ such that the $G$-module $\cO(Z)$ has prescribed
(finite) multiplicities. If $G$ is diagonalizable with character group $\Lambda$, then $Z$ corresponds to
a homogeneous ideal of the polynomial ring $\cO(V)$ for the $\Lambda$-grading defined by the $G$-action; 
we thus recover the multigraded Hilbert scheme. But actually, the construction of the invariant Hilbert 
scheme in \cite{AB05} relies on a reduction to the multigraded case via highest weight theory.

The Hilbert scheme of $\bP^N$ is known to be projective and connected; invariant Hilbert schemes are 
quasi-projective (in particular, of finite type) but possibly non-projective. Also, they may be disconnected, 
already for certain toric Hilbert schemes (see \cite{Sa05}). One may wonder how such moduli schemes can exist 
in the setting of affine algebraic geometry, since for example any curve in the affine plane is a flat limit 
of curves of higher degree. In fact, the condition for the considered subschemes $Z \subset X$ to have a
coordinate ring with finite multiplicities is quite strong; for example, it yields a proper morphism 
to a punctual Hilbert scheme, that associates with $Z$ the categorical quotient $Z/\!/G \subset X/\!/G$. 

In the present article, we expose the construction and fundamental properties of the invariant Hilbert scheme, 
and survey some of its applications to varieties with algebraic group actions. The prerequisites are (hopefully)
quite modest: basic notions of algebraic geometry; the definitions and results that we need about actions and 
representations of algebraic groups are reviewed at the beginning of Section 2. Then we introduce flat families of 
closed $G$-stable subschemes of a fixed affine $G$-scheme $X$, where $G$ is a fixed reductive group, and define
the Hilbert function which encodes the $G$-module structure of coordinate rings. Given such a function $h$,
our main result asserts the existence of a universal family with base a quasi-projective scheme: the invariant 
Hilbert scheme $\Hilb^G_h(X)$.

Section 3 presents a number of basic properties of invariant Hilbert schemes. We first reduce the construction 
of $\Hilb^G_h(X)$ to the case (treated by Haiman and Sturmfels) that $G$ is diagonalizable and $X$ is a $G$-module. 
For this, we slightly generalize the approach of \cite{AB05}, where $G$ was assumed to be connected. Then we 
describe the Zariski tangent space at any closed point $Z$, with special attention to the case that $Z$ is a 
$G$-orbit closure in view of its many applications. Here also, we adapt the results of \cite{AB05}. More original 
are the next developments on the action of the equivariant automorphism group and on a generalization of the 
Hilbert-Chow morphism, which aim at filling some gaps in the literature. 

In Section 4, we first give a brief overview of invariant Hilbert schemes for finite groups and their 
applications to resolving quotient singularities; here the reader may consult \cite{Be08} for a detailed 
exposition. Then we survey very different applications of invariant Hilbert schemes, namely, to the 
classification of \emph{spherical varieties}. These form a remarkable class of varieties equipped with 
an action of a connected reductive group $G$, that includes toric varieties, flag varieties and symmetric 
homogeneous spaces. A normal affine $G$-variety $Z$ is spherical if and only if the $G$-module $\cO(Z)$ 
is multiplicity-free; then $Z$ admits an equivariant degeneration to an affine spherical variety 
$Z_0$ with a simpler structure, e.g., the decomposition of $\cO(Z_0)$ into simple $G$-modules is a grading 
of that algebra. Thus, $Z_0$ is uniquely determined by the highest weights of these simple modules, 
which form a finitely generated monoid $\Gamma$. We show (after \cite{AB05}) that the affine spherical 
$G$-varieties with weight monoid $\Gamma$ are parametrized by an affine scheme of finite type $\M_{\Gamma}$, 
a locally closed subscheme of some invariant Hilbert scheme.

Each subsection ends with examples which illustrate its main notions and results; some of these examples 
have independent interest and are developed along the whole text. In particular, we present results of 
Jansou (see \cite{Ja07}) that completely describe invariant Hilbert schemes associated with the 
``simplest'' data: $G$ is semi-simple, $X$ is a simple $G$-module, and the Hilbert function $h$ is that 
of the cone of highest weight vectors (the affine cone over the closed $G$-orbit $Y$ in the projective space 
$\bP(X)$). Quite remarkably, $\Hilb^G_h(X)$ turns out to be either a (reduced) point or an affine line 
$\bA^1$; moreover, the latter case occurs precisely when $Y$ can be embedded into another projective 
variety $\tY$ as an ample divisor, where $\tY$ is homogeneous under a semi-simple group $\tG \supset G$, 
and $G$ acts transitively on the complement $\tY \setminus Y$. Then the universal family is just the 
affine cone over $\tY$ embedded via the sections of $Y$. 

This relation between invariant Hilbert schemes and projective geometry has been further developed in 
recent works on the classification of arbitrary spherical $G$-varieties. In loose words, one reduces 
to classifying \emph{wonderful $G$-varieties} which are smooth projective $G$-varieties having the 
simplest orbit structure; examples include the just considered $G$-varieties $\tY$. Taking an appropriate 
affine multi-cone, one associates to each wonderful variety a family of affine spherical varieties over 
an affine space $\bA^r$, which turns out to be the universal family. This approach is presented in more 
details in the final Subsections \ref{subsec:fpsv} and \ref{subsec:cwv}.

Throughout this text, the emphasis is on geometric methods and very little space is devoted to 
the combinatorial and Lie-theoretical aspects of the domain, which are however quite important. 
The reader may consult \cite{Bi10, Ch08, MT02, Sa05, St96} for the combinatorics of toric Hilbert 
schemes, \cite{Pe10} for the classification of spherical embeddings, \cite{Bra10} for that of 
wonderful varieties, and \cite{Lo10} for uniqueness properties of spherical varieties. Also, we do 
not present the close relations between certain invariant Hilbert schemes and moduli of polarized 
projective schemes with algebraic group action; see \cite{Al02} for the toric case (and, more generally, 
semiabelic varieties), and \cite{AB06} for the spherical case. These relations would deserve further
investigation.

Also, it would be interesting to obtain a modular interpretation of the \emph{main component}
of certain invariant Hilbert schemes, that contains the irreducible varieties. For toric Hilbert 
schemes, such an interpretation is due to Olsson (see \cite{Ol08}) in terms of logarithmic structures.

Finally, another interesting line of investigation concerns the moduli scheme $\M_{\Gamma}$ 
of affine spherical varieties with weight monoid $\Gamma$. In all known examples, the irreducible 
components of $\M_{\Gamma}$ are just affine spaces, and it is tempting to conjecture that this always 
holds. A positive answer would yield insight into the classification of spherical varieties and 
the multiplication in their coordinate rings.

\medskip

\noindent
{\bf Acknowledgements.} I thank St\'ephanie Cupit-Foutou, Ronan Terpereau, and the referee for their 
careful reading and valuable comments.

\section{Families of affine schemes with reductive group action}
\label{sec:fas}

\subsection{Algebraic group actions}
\label{subsec:aga}

In this subsection, we briefly review some basic notions about algebraic groups and their actions;
details and proofs of the stated results may be found e.g. in the notes \cite{Bri10}. 
We begin by fixing notation and conventions.

Throughout this article, we consider algebraic varieties, algebraic groups, and 
schemes over a fixed algebraically closed field $k$ of characteristic zero. 
Unless explicitly mentioned, schemes are assumed to be \emph{noetherian}.

A {\bf variety} is a reduced separated scheme of finite type; thus,
varieties need not be irreducible. By a point of such a variety $X$, 
we mean a closed point, or equivalently a $k$-rational point.

An {\bf algebraic group} is a variety $G$ equipped with morphisms
$\mu_G : G \times G \to G$ (the multiplication), $\iota_G : G \to G$ 
(the inverse) and with a point $e_G$ (the neutral element)
that satisfy the axioms of a group; this translates into the commutativity
of certain diagrams.

Examples of algebraic groups include closed subgroups of the general linear group $\GL_n$
consisting of all $n \times n$ invertible matrices, where $n$ is a positive 
integer; such algebraic groups are called {\bf linear}. 
We will only consider linear algebraic groups in the sequel.

An (algebraic) {\bf action} of an algebraic group $G$ on a scheme $X$ 
is a morphism 
$$
\alpha : G \times X \longrightarrow X, \quad 
(g,x) \longmapsto g \cdot x
$$
such that the composite morphism
$$
\CD
X @>{e_G \times  \id_X}>> G \times X @>{\alpha}>> X
\endCD
$$
is the identity (i.e., $e_G$ acts on $X$ via the identity $\id_X$), 
and the square
\begin{equation}\label{eqn:asso}
\CD
G \times G \times X @>{\id_G \times \alpha}>> G \times X \\
@V{\mu_G \times \id_X}VV @V{\alpha}VV \\
G \times X @>{\alpha}>> X \\
\endCD
\end{equation}
commutes (the associativity property of an action). 

A scheme equipped with a $G$-action is called a $G$-{\bf scheme}.
Given two $G$-schemes $X$, $Y$ with action morphisms $\alpha$, $\beta$,
a morphism $f : X \to Y$ is called {\bf equivariant} 
(or a {\bf $G$-morphism}) if the square
$$
\CD
G \times X @>{\alpha}>> X\\
@V{\id_G \times f}VV @V{f}VV \\
G \times Y @>{\beta}>> Y\\
\endCD
$$
commutes. If $\beta$ is the trivial action, i.e., the projection $G \times Y \to Y$, 
then we say that $f$ is $G$-{\bf invariant}.

An (algebraic, or rational) $G$-{\bf module} is a vector space $V$ over $k$, 
equipped with a linear action of $G$ which is algebraic in the following sense: 
every $v \in V$ is contained in a $G$-stable finite-dimensional subspace $V_v \subset V$
on which $G$ acts algebraically. (We do not assume that $V$ itself is 
finite-dimensional). Equivalently, $G$ acts on $V$ via a representation
$\rho : G \to \GL(V)$ which is a union of finite-dimensional algebraic 
subrepresentations.

A $G$-stable subspace $W$ of a $G$-module $V$ is called a 
$G$-{\bf submodule}; $V$ is {\bf simple} if it has no non-zero proper
submodule. Note that simple modules are finite-dimensional, and 
correspond to irreducible representations.

A $G$-{\bf algebra} is an algebra $A$ over $k$, equipped with an action 
of $G$ by algebra automorphisms which also makes it a $G$-module.

Given a $G$-scheme $X$, the algebra $\cO(X)$ of global sections of 
the structure sheaf is equipped with a linear action of $G$ via
$$
(g \cdot f)(x) = f(g^{-1} \cdot x).
$$
In fact, this action is algebraic and hence \emph{$\cO(X)$ is a $G$-algebra}.
The assignement $X \mapsto \cO(X)$ defines a bijective correspondence 
between \emph{affine} $G$-schemes and $G$-algebras. 

Each $G$-algebra of finite type $A$ is generated by a finite-dimensional $G$-submodule $V$.
Hence $A$ is a quotient of the symmetric algebra $\Sym(V) \cong \cO(V^*)$ by a $G$-stable ideal. 
It follows that \emph{every affine $G$-scheme of finite type is isomorphic to a closed 
$G$-subscheme of a finite-dimensional $G$-module}.

\begin{examples}\label{ex:aga}
(i) Let $G := \GL_1$ be the {\bf multiplicative group}, denoted by $\bG_m$. Then 
$\cO(G) \cong k[t,t^{-1}]$ and the $G$-modules are exactly the graded vector spaces
\begin{equation}\label{eqn:grad}
V = \bigoplus_{n \in \bZ} V_n
\end{equation}
where $\bG_m$ acts via $t \cdot \sum_n v_n = \sum_n t^n v_n$. In particular,
the $G$-algebras are just the $\bZ$-graded algebras.

\noindent
(ii) More generally, consider a {\bf diagonalizable group} $G$, i.e., a closed 
subgroup of a {\bf torus} $(\bG_m)^n$. Then $G$ is uniquely determined
by its {\bf character group} $\cX(G)$, the set of homomorphisms
$\chi : G \to \bG_m$ equipped with pointwise multiplication. Moreover, 
the abelian group $\cX(G)$ is finitely generated, and the assignement
$G \mapsto \cX(G)$ defines a bijective correspondence between
diagonalizable groups and finitely generated abelian groups. This correspondence
is contravariant, and $G$-modules correspond to $\cX(G)$-graded vector spaces.

Any character of $\bG_m$ is of the form $t \mapsto t^n$ for some integer $n$; 
this identifies $\cX(\bG_m)$ with $\bZ$.
\end{examples}

\subsection{Reductive groups}
\label{subsec:rg} 

In this subsection, we present some basic results on reductive groups, their representations
and invariants; again, we refer to the notes \cite{Bri10} for details and proofs.

A linear algebraic group $G$ is called {\bf reductive} if every $G$-module is semi-simple, 
i.e., isomorphic to a direct sum of simple $G$-modules. In view of the characteristic-$0$ 
assumption, this is equivalent to the condition that $G$ has no non-trivial closed normal 
subgroup isomorphic to the additive group of a finite-dimensional vector space; 
this is the group-theoretical notion of reductivity.
 
Examples of reductive groups include finite groups, diagonalizable groups, 
and the classical groups such as $\GL_n$ and the special linear group
$\SL_n$ (consisting of $n \times n$ matrices of determinant $1$).

Given a reductive group $G$, we denote by $\Irr(G)$ the set of isomorphism 
classes of simple $G$-modules (or of irreducible $G$-representations). 
The class of the trivial $G$-module $k$ is denoted by $0$. 

For any $G$-module $V$, the map 
$$
\bigoplus_{M \in \Irr(G)} \Hom^G(M,V) \otimes_k M \longrightarrow V,
\quad \sum_M f_M \otimes x_M \longmapsto \sum_M f_M(x_M)
$$
is an isomorphism of $G$-modules, where $\Hom^G(M,V)$ denotes the vector space of 
morphisms of $G$-modules from $M$ to $V$, and $G$ acts on the left-hand side 
via
$$
g \cdot \sum_M f_M \otimes x_M = \sum_M f_M \otimes g \cdot x_M.
$$
Thus, the dimension of $\Hom^G(M,V)$ is the multiplicity of $M$ in $V$ 
(which may be infinite). This yields the {\bf isotypical decomposition}
\begin{equation}\label{eqn:decm}
V \cong  \bigoplus_{M \in \Irr(G)} V_M \otimes_k M \quad \text{where} \quad
V_M := \Hom^G(M,V).
\end{equation}
In particular, $V_0$ is the subspace of $G$-invariants in $V$, denoted by $V^G$.

For a $G$-algebra $A$, the invariant subspace $A^G$ is a subalgebra. Moreover, 
each $A_M$ is an $A^G$-module called the {\bf module of covariants of type $M$}.
Denoting by $X = \Spec(A)$ the associated $G$-scheme, we also have 
$$
A_M \cong \Mor^G(X,M^*)
$$  
(the set of $G$-morphisms from $X$ to the dual module $M^*$; note that $M^*$ is simple).
Also, we have an isomorphism of $A^G$-$G$-modules in an obvious sense
\begin{equation}\label{eqn:cova}
A \cong \bigoplus_{M \in \Irr(G)} A_M \otimes_k M.
\end{equation}

\begin{example}\label{ex:ag}
Let $G$ be a diagonalizable group with character group $\Lambda$. Then $G$ is reductive
and its simple modules are exactly the lines $k$ where $G$ acts via a character 
$\lambda \in \Lambda$. The decompositions (\ref{eqn:decm}) and (\ref{eqn:cova}) give back
the $\Lambda$-gradings of $G$-modules and $G$-algebras. 
\end{example}

Returning to an arbitrary reductive group $G$, the modules of covariants satisfy 
an important finiteness property:

\begin{lemma}\label{lem:finG}
Let $A$ be a $G$-algebra, finitely generated over a subalgebra 
$R \subset A^G$. Then the algebra $A^G$ is also finitely generated
over $R$. Moreover, the $A^G$-module $A_M$ is finitely generated for
any $M \in \Irr(G)$.
\end{lemma}

In the case that $R = k$, this statement is a version of the classical 
Hilbert-Nagata theorem, see e.g. \cite[Theorem 1.25 and Lemma 2.1]{Bri10}.
The proof given there adapts readily to our setting, which will be generalized 
to the setting of families in Subsection \ref{subsec:fam}. 

In particular, for an affine $G$-scheme of finite type $X = \Spec(A)$,
the subalgebra $A^G$ is finitely generated and hence corresponds to
an affine scheme of finite type $X/\!/G$, equipped with a $G$-invariant morphism
$$
\pi : X \longrightarrow X/\!/G.
$$
In fact, $\pi$ is universal among all invariant morphisms with source $X$;
thus, $\pi$ is called the {\bf categorical quotient}. Also, given a closed $G$-subscheme
$Y \subset X$, the restriction $\pi_{\vert Y}$ is the categorical quotient of $Y$, and 
the image of $\pi_{\vert Y}$ is closed in $X/\!/G$. As a consequence, $\pi$ is surjective
and each fiber contains a unique closed $G$-orbit.

We now assume $G$ to be \emph{connected}.
Then the simple $G$-modules may be explicitly described via highest weight theory
that we briefly review. Choose a {\bf Borel subgroup} $B \subset G$, i.e., 
a maximal connected solvable subgroup. Every simple $G$-module $V$ contains 
a unique line of eigenvectors of $B$, say $k v$; then $v$ is called a {\bf highest weight vector}. 
The corresponding character $\lambda : B \to \bG_m$, such that $b \cdot v  = \lambda(b) v$
for all $b \in B$, is called the {\bf highest weight of $V$}; it uniquely determines 
the $G$-module $V$ up to isomorphism. We thus write $V = V(\lambda)$, $v = v_{\lambda}$, 
and identify $\Irr(G)$ with a subset of the character group $\Lambda := \cX(B)$: 
the set $\Lambda^+$ of {\bf dominant weights}, which turns out to be a finitely generated 
submonoid of the {\bf weight lattice} $\Lambda$. Moreover, $V(0)$ is the trivial $G$-module $k$.

In this setting, the modules of covariants admit an alternative description in terms of 
highest weights. To state it, denote by $U$ the unipotent part of $B$; this is a closed 
connected normal subgroup of $B$, and a maximal unipotent subgroup of $G$. 
Moreover, $B$ is the semi-direct product $UT$ where $T \subset B$ is a maximal torus, 
and $\Lambda$ is identified with the character group $\cX(T)$ via the restriction. 
Given a $G$-module $V$, the subspace of $U$-invariants $V^U$ is a $T$-module and hence 
a direct sum of eigenspaces $V^U_{\lambda}$ where $\lambda \in \Lambda$. Moreover, the map
$$
\Hom^G\big( V({\lambda}), V) \longrightarrow V^U_{\lambda}, \quad 
f \longmapsto f(v_\lambda)
$$
is an isomorphism for any $\lambda \in \Lambda^+$, and $V^U_{\lambda} = 0$
for any $\lambda \notin \Lambda^+$. As a consequence, given a $G$-algebra $A$,
we have isomorphisms 
$$
A_{V(\lambda)} \cong A^U_{\lambda}
$$ 
of modules of covariants over 
$$
A^G = A^B = A^U_0.
$$

The algebra of $U$-invariants also satisfies a useful finiteness property
(see e.g. \cite[Theorem 2.7]{Bri10}):

\begin{lemma}\label{lem:finU}
With the notation and assumptions of Lemma \ref{lem:finG}, the algebra
$A^U$ is finitely generated over $R$.
\end{lemma}

This in turn yields a categorical quotient
$$
\psi : X \longrightarrow X/\!/U
$$
for an affine $G$-scheme of finite type $X$, where $X/\!/U := \Spec\, \cO(X)^U$
is an affine $T$-scheme of finite type. Moreover, for any closed $G$-subscheme 
$Y \subset X$, the restriction $\psi_{\vert Y}$ is again the categorical quotient of $Y$
(but $\psi$ is generally not surjective).

Also, many properties of $X$ may be read off $X/\!/U$. For example, 
\emph{an affine $G$-scheme $X$ is of finite type (resp. reduced, irreducible, normal) 
if and only if so is $X/\!/U$} (see \cite[Chapter 18]{Gr97}).

\begin{examples}\label{ex:agbis}
(i) If $G$ is a torus, then $G = B = T$ and $\Lambda^+ = \Lambda$; each $V(\lambda)$
is just the line $k$ where $T$ acts via $t \cdot z = \lambda(t) z$.

\medskip

\noindent
(ii) Let $G = \SL_2$. We may take for $B$ the subgroup of upper triangular matrices with
determinant $1$. Then $U$ is the subgroup of upper triangular matrices with diagonal entries 
$1$, isomorphic to the additive group $\bG_a$. Moreover, we may take for $T$ the subgroup 
of diagonal matrices with determinant $1$, isomorphic to the multiplicative group $\bG_m$. 
Thus, $\Lambda \cong \bZ$. 

In fact, the simple $G$-modules are exactly the spaces $V(n)$ of homogeneous polynomials 
of degree $n$ in two variables $x,y$ where $G$ acts by linear change of variables;
here $n \in \bN \cong \Lambda^+$. A highest weight vector in $V(n)$ is the monomial $y^n$.
Moreover, $V(n)$ is isomorphic to its dual module, and hence to the $n$-th symmetric
power $\Sym^n(k^2)$ where $k^2 \cong V(1)^* \cong V(1)$ denotes the standard $G$-module.

Since $G$ acts transitively on $V(1) \setminus \{0\}$, the categorical quotient
$V(1)/\!/G$ is just a point. One easily shows that the quotient by $U$ is the map
$$
V(1) \longrightarrow \bA^1, \quad a x + b y \longmapsto a.
$$

Also, the categorical quotient of $V(2)$ by $G$ is given by the discriminant
$$
\Delta : V(2) \longrightarrow \bA^1, \quad a x^2 + 2 b xy + c y^2 \longmapsto ac - b^2
$$
and the categorical quotient by $U$ is the map
$$
V(2) \longrightarrow \bA^2, \quad a x^2 + 2 b xy + c y^2 \longmapsto (a,ac - b^2).
$$

For large $n$, no explicit description of the categorical quotients $V(n)/\!/G$ and 
$V(n)/\!/U$ is known, although the corresponding algebras $\Sym\big( V(n) \big)^G$
and $\Sym\big( V(n) \big)^U$ (the invariants and covariants of binary forms of degree $n$)
have been extensively studied since the mid-19th century.

\end{examples}

\subsection{Families}
\label{subsec:fam}

In this subsection, we fix a reductive group $G$ and an affine $G$-scheme of finite type, 
$X = \Spec(A)$. We now introduce our main objects of interest. 

\begin{definition}\label{def:fam}
A {\bf family of closed $G$-subschemes of $X$ over a scheme $S$}
is a closed subscheme $\cZ \subset X \times S$, stable by the 
action of $G$ on $X \times S$ via $g \cdot (x,s) = (g \cdot x, s)$. 

The family $\cZ$ is {\bf flat}, if so is the morphism $p : \cZ \to S$
induced by the projection $p_2 : X \times S \to S$.
\end{definition}

Given a family $\cZ$ as above, the morphism $p$ is $G$-invariant; thus, for any $k$-rational 
point $s$ of $S$, the scheme-theoretic fiber $\cZ_s$ of $p$ at $s$ is a closed 
$G$-subscheme of $X$. More generally, an arbitrary point $s \in S$ yields a closed subscheme 
$\cZ_{k(\bar{s})} \subset X_{k(\bar{s})}$ where $k(s)$ denotes the residue field of $s$, 
and $k(\bar{s})$ is an algebraic closure of that field. The scheme $\cZ_{k(\bar{s})}$ is 
the geometric fiber of $\cZ$ at $s$, also denoted by $\cZ_{\bar{s}}$; this is a closed
subscheme of $X_{\bar{s}}$ equipped with an action of $G$ (viewed as a constant group scheme)
or equivalently of $G_{\bar{s}}$ (the group of $k(\bar{s})$-rational points of~$G$).

Since $X$ is affine, the data of the family $\cZ$ is equivalent to 
that of the sheaf $p_* \cO_{\cZ}$ as a quotient of 
$(p_2)_* \cO_{X \times S} = A \otimes_k \cO_S$, where both are viewed as sheaves 
of $\cO_S$-$G$-algebras. Moreover, as $G$ acts trivially on $S$, we have 
a canonical decomposition
\begin{equation}\label{eqn:decs}
p_* \cO_{\cZ} \cong \bigoplus_{M \in \Irr(G)} \cF_M \otimes_k M,
\end{equation}
where each {\bf sheaf of covariants}
\begin{equation}\label{eqn:covs}
\cF_M := \Hom^G(M, p_*\cO_Z) = ( p_*\cO_Z \otimes_k M^*)^G 
\end{equation}
is a sheaf of $\cO_S$-modules, and (\ref{eqn:decs}) is an 
isomorphism of sheaves of $\cO_S$-$G$-modules.

Also, $p_* \cO_{\cZ}$ is a sheaf of finitely generated $\cO_S$-algebras,
since $X$ is of finite type. In view of Lemma \ref{lem:finG}, it follows that 
$$
\cF_0 = (p_* \cO_{\cZ})^G
$$ 
is a sheaf of finitely generated $\cO_S$-algebras as well; moreover, $\cF_M$ 
is a coherent sheaf of $\cF_0$-modules, for any $M \in \Irr(G)$. By (\ref{eqn:decs}), 
the family $\cZ$ is flat if and only if each sheaf of covariants $\cF_M$ is flat.

\begin{definition}\label{def:fm}
With the preceding notation, the family $\cZ$ is {\bf multiplicity-finite} 
if the sheaf of $\cO_S$-modules $\cF_M$ is coherent for any $M \in \Irr(G)$; 
equivalently, $\cF_0$ is coherent.

We say that $\cZ$ is {\bf multiplicity-free} if each $\cF_M$ is zero or invertible.
\end{definition}

Since flatness is equivalent to local freeness for a finitely generated 
module over a noetherian ring, we see that \emph{a multiplicity-finite
family is flat iff each sheaf of covariants is locally free of finite rank}.
When the base $S$ is connected, the ranks of these sheaves are well-defined 
and yield a numerical invariant of the family: the {\bf Hilbert function}
$$
h = h_{\cZ} : \Irr(G) \longrightarrow \bN, \quad M \longmapsto \rk_{\cO_S}(\cF_M).
$$ 

This motivates the following: 

\begin{definition}\label{def:hilb}
Given a function $h : \Irr(G) \to \bN$,
a {\bf flat family of closed subschemes of $X$ with Hilbert function $h$} is 
a closed subscheme $\cZ \subset X \times S$ such that each sheaf of covariants 
$\cF_M$ is locally free of rank $h(M)$.
\end{definition}

\begin{remarks}\label{rem:fam}
(i) In the case that $S = \Spec(k)$, a family is just a closed $G$-subscheme 
$Z \subset X$. Then $Z$ is multiplicity-finite
if and only if the quotient $Z/\!/G$ is finite; equivalently, $Z$ contains only
finitely many closed $G$-orbits. For example, \emph{any $G$-orbit closure is multiplicity-finite}.

Also, $Z$ is multiplicity-free if and only if the $G$-module $\cO(Z)$ has multiplicities 
$0$ or $1$. If $Z$ is an irreducible variety, this is equivalent to the condition that 
$Z$ contains an open orbit of a Borel subgroup $B \subset G$ (see e.g. \cite[Lemma 2.12]{Bri10}). 

In particular, when $G$ is a torus, say $T$, an affine irreducible $T$-variety $Z$ is multiplicity-free 
if and only if it contains an open $T$-orbit. Then each non-zero eigenspace $\cO(Z)_{\lambda}$ is a line, 
and the set of those $\lambda$ such that $\cO(Z)_{\lambda}\neq 0$ consists of those linear combinations
$n_1 \lambda_1 + \cdots n_r \lambda_r$, where $n_1,\ldots,n_r$ are non-negative integers,
and $\lambda_1, \ldots, \lambda_r$ are the weights of homogeneous generators of the algebra
$\cO(Z)$. Thus, this set is a finitely generated submonoid of $\Lambda$
that we denote by $\Lambda^+(Z)$ and call the {\bf weight monoid of $Z$}.

Each affine irreducible multiplicity-free $T$-variety $Z$ is uniquely determined by its weight monoid $\Gamma$: 
in fact, \emph{the $\Lambda$-graded algebra $\cO(Z)$ is isomorphic to $k[\Gamma]$}, the algebra of the monoid 
$\Gamma$ over $k$. Moreover, $Z$ is normal if and only if $\Gamma$ is {\bf saturated}, i.e., equals the 
intersection of the group that it generates in $\Lambda$, with the convex cone that it generates in 
$\Lambda \otimes_{\bZ} \bR$. Under that assumption, $Z$ is called an (affine) {\bf toric variety}.

Returning to an affine irreducible $G$-variety $Z$, note that $Z$ is multiplicity-free if
and only if so is $Z/\!/U$. In that case, we have an isomorphism of $G$-modules
\begin{equation}\label{eqn:wm}
\cO(Z) \cong \bigoplus_{\lambda \in \Lambda^+(Z)} V(\lambda)
\end{equation}
where $\Lambda^+(Z) := \Lambda^+(Z/\!/U)$ is again called the {\bf weight monoid of $Z$}.
In other words, the Hilbert function of $Z$ is given by
\begin{equation}\label{eqn:1-0}
h(\lambda) = 
\begin{cases}
1 & \text{if $\lambda \in \Lambda^+(Z)$,} \\
0 & \text{otherwise.}\\
\end{cases}
\end{equation}

Also, \emph{$Z$ is normal if and only if $\Lambda^+(Z)$ is saturated}; then $Z$ is called
an (affine) {\bf spherical $G$-variety}. In contrast to the toric case, spherical varieties
are not uniquely determined by their weight monoid, see e.g. Example \ref{ex:fam}(ii).

\medskip

\noindent
(ii) A flat family $\cZ$ over a \emph{connected} scheme $S$ is multiplicity-finite
(resp. is multiplicity-free, or has a prescribed Hilbert function $h$) if and only if 
so does some geometric fiber $\cZ_{\bar{s}}$. In particular, if some fiber is a spherical
variety, then the family is multiplicity-free.

\medskip

\noindent
(iii) Any family of closed $G$-subschemes $\cZ \subset X \times S$ yields a family
of closed $T$-subschemes $\cZ/\!/U \subset X/\!/U \times S$; moreover, the sheaves of
covariants of $\cZ$ and $\cZ/\!/U$ are isomorphic. Thus, $\cZ$ is flat 
(multiplicity-finite, multiplicity-free) if and only if so is $\cZ/\!/U$. Also,
$\cZ$ has Hilbert function $h$ if and only if $\cZ/\!/U$ has Hilbert function $\bar{h}$
such that
\begin{equation}\label{eqn:bh}
\bar{h}(\lambda) = \begin{cases} 
h(\lambda) & \text{if $\lambda \in \Lambda^+$}, \\
0 & \text{otherwise}. \\
\end{cases}
\end{equation} 

\end{remarks} 

\begin{examples}\label{ex:fam}
(i) The surface $\cZ$ of equation $x y - z = 0$ in $\bA^3$ is stable under the action of 
$\bG_m$ via $t \cdot (x,y,z) = (tx, t^{-1}y, z)$. The morphism $z : \cZ \to \bA^1$ is a flat family
of closed $\bG_m$-subschemes of $\bA^2$ (the affine plane with coordinates $x,y$). The fibers
over non-zero points of $\bA^1$ are all isomorphic to $\bG_m$; they are exactly the orbits of
points of $\bA^2$ minus the coordinate axes. In particular, the family has Hilbert function
the constant function $1$. The fiber at $0$ is the (reduced) union of the coordinate axes.

For the $\bG_m$-action on $\bA^3$ via $t \cdot (x,y,z) = (t^2 x, t^{-1}y, z)$, the surface
$\cW$ of equation $x y^2 = z$ yields a family with the same fibers at non-zero points, but 
the fiber at $0$ is non-reduced.

More generally, consider a torus $T$ acting linearly on the affine space $V = \bA^N$
with pairwise distinct weights. Denote by $\lambda_1,\ldots,\lambda_N$ the opposites of these
weights, i.e., the weights of the coordinate functions. Also, let $v \in V$ be 
a {\bf general point} in the sense that all its coordinates are non-zero. Then the orbit closure
$$
Z := \overline{T \cdot v} \subset V
$$
is an irreducible multiplicity-free variety, and different choices of $v$ yield isomorphic 
$T$-varieties; moreover, all irreducible multiplicity-free varieties may be obtained in this way. 
The weight monoid of $Z$ is generated by $\lambda_1,\ldots,\lambda_N$.

We construct flat families over $\bA^1$ with general fiber $Z$ as follows. 
Let the torus $T \times \bG_m$ act linearly on $V \times \bA^1 = \bA^{N+1}$ 
such that the coordinate functions have weights
$$
(\lambda_1,a_1), \ldots, (\lambda_N,a_N), (0,1)
$$
where $a_1,\ldots,a_N$ are integers, viewed as characters of $\bG_m$. Then the orbit closure
$$
\cZ := \overline{(T \times \bG_m) \cdot (v,1)} \subset V \times \bA^1
$$
may be viewed as a $T$-variety. The projection $p: \cZ \to \bA^1$ is $T$-invariant,
and flat since $\cZ$ is an irreducible variety. Moreover, $p$ is trivial over 
$\bA^1 \setminus \{0\}$; specifically, the map
$$ 
p^{-1}(\bA^1 \setminus \{0\}) \longrightarrow Z \times (\bA^1 \setminus \{0\}),
\quad (v,s) \longmapsto (s^{-1}v, s)
$$
is a $T$-equivariant isomorphism of families over $\bA^1 \setminus \{0\}$.
In particular, the fibers of $p$ at non-zero points are all isomorphic to $Z$,
and $p$ is multiplicity-free with Hilbert function $h$ as above.
On the other hand, the special fiber $\cZ_0$ is non-empty if and only if
$p$ (viewed as a regular function on $\cZ$) is not invertible; this translates
into the condition that the convex cone generated by 
$(\lambda_1,a_1), \ldots, (\lambda_N,a_N)$ does not contain $(0,-1)$.

One may show that the preceding construction yields all one-parameter families with 
generic fiber a multiplicity-free variety. Also, one may show that the special fiber is reducible 
unless the whole family is trivial; this contrasts with our next example, where all fibers are 
irreducible varieties and the special fiber is singular while all others are smooth.

\medskip

\noindent
(ii) Let $G = \SL_2$ and $\Delta : V(2) \to \bA^1$ the discriminant as in Example
\ref{ex:agbis}(ii). Then the graph
$$
\cZ = \{(f,s) \in  V(2) \times \bA^1 ~\vert~ \Delta(f) = s \}
$$
is a flat family of $G$-stable closed subschemes of $V(2)$. The fibers at non-zero 
closed points are exactly the $G$-orbits of non-degenerate quadratic forms, while the 
fiber at $0$ consists of two orbits: the squares of non-zero linear forms, and the origin. 
Since $\cZ \cong V(2)$ as $G$-varieties, the Hilbert function of $\cZ$ is given by
$$
h_2(n) = 
\begin{cases} 
1 & \text{if $n$ is even}, \\
0 & \text{if $n$ is odd}.\\
\end{cases}
$$
Thus, $\cZ$ is multiplicity-free. Moreover, the family $\cZ/\!/U$ is trivial with fiber $\bA^1$,
as follows from the description of $V(2)/\!/U$ in Example \ref{ex:agbis}(ii). In particular, 
each fiber $\cZ_s$ is a spherical variety. 

Next, consider the quotient $\cW$ of $V(2)$ by the involution $\sigma : f \mapsto -f$.
Then $\cW$ is the affine $G$-variety associated with the subalgebra of 
$\cO\big(V(2)\big) \cong \Sym \big( V(2) \big)$ consisting of even polynomial functions,
i.e., the subalgebra generated by 
$\Sym^2 \big( V(2) \big) \cong V(4) \oplus V(0)$. In other words, 
$$
\cW \subset V(4) \times \bA^1 
$$
and the resulting projection $q : \cW \to \bA^1$ may be identified with the $\sigma$-invariant 
map $V(2) \to \bA^1$ given by the discriminant. It follows that $q$ is a flat family of 
$G$-stable closed subschemes of $V(4)$, with Hilbert function given by
$$
h_4(n) = 
\begin{cases} 
1 & \text{if $n$ is a multiple of $4$}, \\
0 & \text{otherwise}.\\
\end{cases}
$$
In particular, $\cW$ is multiplicity-free. Moreover, its fibers at non-zero closed points are 
exactly the orbits $G \cdot f^2 \subset V(4)$, where $f$ is a non-degenerate quadratic form,
while the fiber at $0$ consists of two orbits: the fourth powers of non-zero linear forms, 
and the origin. The family $\cW/\!/U$ is again trivial with fiber $\bA^1$, so that the fibers
of $\cW$ are spherical varieties. 

We will show that both families just constructed are universal (in the sense of the next
subsection), and that no family with similar properties exists in $V(n)$ for $n \neq 2,4$. 
For this, we will apply the various techniques that we successively introduce; 
see Examples \ref{ex:tuf}(ii), \ref{ex:zts}(ii), \ref{ex:aut}(ii) and \ref{ex:qsm}(ii). 
\end{examples}

\subsection{The universal family}
\label{subsec:tuf}

In the setting of the previous subsection, there is a natural construction of {\bf pull-back} 
for families of $G$-stable subschemes of $X$: given such a family $\cZ \subset X \times S$ 
and a morphism of schemes $f : S' \to S$, we can form the cartesian square
$$
\CD
\cZ' @>>> X \times S' \\
@VVV @V{\id_X \times f}VV \\
\cZ @>>> X \times S \\
\endCD
$$ 
where the horizontal arrows are inclusions. This yields a family of closed
$G$-subschemes of $X$ over $S'$: the pull-back of $\cZ$ under $f$, which may also be defined
via the cartesian square
$$
\CD
\cZ' @>>> S' \\
@VVV @V{f}VV \\
\cZ @>{p}>> S. \\
\endCD
$$ 
Note that $\cZ'$ is flat over $S'$ whenever $\cZ$ is flat over $S$; moreover, multiplicity-finiteness
and -freeness are preserved under pull-back, as well as the Hilbert function.

We may now state our main result, which asserts the existence of a universal family:

\begin{theorem}\label{thm:hilb}
Given a reductive group $G$, an affine $G$-scheme of finite type $X$ and a function 
$h: \Irr(G) \to \bN$, there exists a family of closed $G$-subschemes with Hilbert function $h$,
\begin{equation}\label{eqn:univ}
\Univ^G_h(X) \subset X \times \Hilb^G_h(X),
\end{equation}
such that any family $\cZ \subset X \times S$ of closed $G$-subschemes with Hilbert 
function $h$ is obtained from (\ref{eqn:univ}) by pull-back under 
a unique morphism $f : S \to \Hilb^G_h(X)$. Moreover, the scheme $\Hilb^G_h(X)$ is 
quasi-projective (in particular, of finite type).
\end{theorem}

The family (\ref{eqn:univ}) is of course uniquely determined up to a unique isomorphism
by its universal property. The scheme $\Hilb^G_h(X)$ is called the {\bf invariant Hilbert scheme}
associated with the affine $G$-scheme $X$ and the function $h$.

Theorem \ref{thm:hilb} may be reformulated as asserting that
\emph{the {\bf Hilbert functor} $Hilb^G_h(X)$ that associates with any scheme $S$, 
the set of flat families $\cZ \subset X \times S$ with Hilbert function $h$, 
is represented by the quasi-projective scheme $\Hilb^G_h(X)$}. 

By taking $S = \Spec(R)$ where $R$ is an arbitrary algebra, this yields 
an algebraic description of the $R$-points of the invariant Hilbert scheme: 
these are those $G$-stable ideals $I \subset A \otimes_k R$ such that 
each $R$-module of covariants 
$$
\big( (R \otimes_k A)/I \big)_M = (R \otimes_k A)_M/I_M
$$ 
is locally free of rank $h(M)$.

In particular, the $k$-rational points of $\Hilb^G_h(X)$ (which are the same
as its closed points, since this scheme is of finite type) are those
$G$-stable ideals $I$ of $A = \cO(X)$ such that each simple $G$-module $M$ has 
multiplicity $h(M)$ in the quotient $A/I$. These points may also be identified
with the closed $G$-stable subschemes $Z \subset X$ with Hilbert function $h$.

\begin{remarks}\label{rem:tuf}
(i) The case of the {\bf trivial group $G$} is already quite substantial.
There, $X$ is just an affine scheme of finite type, and a family with finite 
multiplicities is exactly a closed subscheme $\cZ \subset X \times S$
such that the projection $p: \cZ \to S$ is finite. Moreover, a Hilbert function
is just a non-negative integer $n$. In that case, $\Hilb_n(X)$ is the
{\bf punctual Hilbert scheme} that parametrizes the closed subschemes of length $n$
of $X$. In fact, $\Hilb_n(X)$ exists more generally for any quasi-projective
scheme $X$ over a field of arbitrary characteristic.

\medskip

\noindent
(ii) If $G$ is the multiplicative group $\bG_m$, we know that $X$ corresponds
to a $\bZ$-graded algebra of finite type $A$. For a Hilbert function $h : \bZ \to \bN$,
the scheme $\Hilb_h^{\bG_m}(X)$ parametrizes those graded ideals $I \subset A$ 
such that the vector space $(A/I)_n$ has dimension $h(n)$ for all $n \in \bZ$. 

Of special interest is the case that $X$ is the affine space $\bA^N$ where $\bG_m$ 
acts via scalar multiplication, i.e., $A$ is a polynomial ring in $N$ indeterminates 
of weight $1$. Then a necessary condition for the existence of such ideals $I$ is that 
$h(n) = P(n)$ for all $n \gg 0$, where $P(t)$ is a (uniquely determined) polynomial:
the {\bf Hilbert polynomial} of the graded algebra $A/I$.
In that case, we also have the {\bf Hilbert scheme} $\Hilb_P(\bP^{N-1})$ that parametrizes 
closed subschemes of the projective $(N-1)$-space with Hilbert polynomial $P$, or equivalently, 
graded ideals $I \subset A$ such that $\dim(A/I)_n = P(n)$ for all $n \gg 0$. 
This yields a morphism 
$$
\Hilb^{\bG_m}_h(\bA^N) \longrightarrow \Hilb_P(\bP^{N-1})
$$ 
which is in fact an isomorphism for an appropriate choice of the Hilbert function $h$,
associated to a given Hilbert polynomial $P$ (see \cite[Lemma 4.1]{HS04}).

\medskip

\noindent
(iii) More generally, if $G$ is diagonalizable with character group $\Lambda$, and
$X$ is a $G$-module of finite dimension $N$, then $A$ is a polynomial ring on homogeneous
generators with weights $\lambda_1,\ldots,\lambda_N \in \Lambda$. Moreover, $\Hilb_h^G(X)$
parametrizes those $\Lambda$-graded ideals $I \subset A$ such that each vector space
$(A/I)_{\lambda}$ has a prescribed dimension $h(\lambda)$. In that case, $\Hilb_h^G(X)$ is the 
{\bf multigraded Hilbert scheme} of \cite{HS04}. As shown there, that scheme exists over any base ring, 
and no noetherian assumption is needed in the definition of the corresponding functor.
\end{remarks}

\begin{examples}\label{ex:tuf}
(i) Consider a torus $T$ acting linearly on the affine space $\bA^N$ via pairwise distinct weights 
and take for $h$ the Hilbert function of a general $T$-orbit closure. Denoting by 
$\lambda_1,\ldots,\lambda_N$ the weights of the coordinate functions on $\bA^N$ 
and by $\Gamma$ the submonoid of $\Lambda$ generated by these weights, we have
\begin{equation}\label{eqn:hm}
h(\lambda) = \begin{cases}
1 & \text{if $\lambda \in \Gamma$}, \\
0 & \text{otherwise}. \\
\end{cases}
\end{equation}
The associated invariant Hilbert scheme $\Hilb^T_h(\bA^N)$ is called the {\bf toric Hilbert scheme}; 
it has been constructed by Peeva and Stillman (see \cite{PS02}) prior to the more general construction 
of multigraded Hilbert schemes. Since $\Hilb^T_h(\bA^N)$ only depends on $T$ and 
$\ulambda = (\lambda_1,\ldots,\lambda_N)$, we will denote it by $\Hilb^T(\ulambda)$. 

\medskip

\noindent
(ii) Let $G = \SL_2$ and take for $X$ the simple $G$-module $V(n)$. Then $X$ contains 
a distinguished closed $G$-stable subvariety $Z$, consisting of the $n$-th powers of 
linear forms. In other words, $Z$ is the affine cone over the image of the $n$-uple
embedding of $\bP^1$ in $\bP^n = \bP\big( V(n) \big)$. Since that image is the unique closed $G$-orbit, 
$Z$ is the smallest non-zero closed $G$-stable subcone of $V(n)$. Also, $Z$ is a normal surface with 
singular locus the origin if $n \geq 2$, while $Z = V(1)$ if $n = 1$. Moreover, the Hilbert function 
of $Z$ is given by
\begin{equation}\label{eqn:hn}
h_n(m) = 
\begin{cases}
1 & \text{if $m$ is a multiple of $n$}, \\
0 & \text{otherwise}. \\
\end{cases}
\end{equation}
As we will show, the corresponding invariant Hilbert scheme is the affine line if $n = 2$ or $4$; 
in both cases, the universal family is that constructed in Example \ref{ex:fam}. For all other values 
of $n$, the invariant Hilbert scheme consists of the (reduced) point $Z$. 

\medskip

\noindent
(iii) More generally, let $G$ be an arbitrary connected reductive group, 
$V= V(\lambda)$ a non-trivial simple $G$-module, and $v = v_{\lambda}$ a highest weight vector. 
Then the corresponding point $[v]$ of the projective space $\bP(V)$ is the unique $B$-fixed point. 
Hence $G \cdot [v]$ is the unique closed $G$-orbit in $\bP(V)$, by Borel's fixed point theorem. Thus, 
$$
Z := \overline{G \cdot v} = G \cdot v \cup \{0\}
$$ 
is the smallest non-zero $G$-stable cone in $V$: the {\bf cone of highest weight vectors}.
Moreover, we have an isomorphism of graded $G$-modules
\begin{equation}\label{eqn:hwc}
\cO(Z) \cong \bigoplus_{n=0}^{\infty} V(n\lambda)^*
\end{equation}
where $V(n\lambda)^*$ has degree $n$. Thus, denoting by $\lambda^*$ the highest weight of the
simple $G$-module $V(\lambda)^*$, we see that the $T$-algebra $\cO(Z)^U$ is a polynomial ring
in one variable of weight $\lambda^*$. In particular, $Z/\!/U$ is normal, and hence $Z$ is a 
spherical variety. Its Hilbert function is given by
$$
h_{\lambda}(\mu) = 
\begin{cases}
1 & \text{if $\mu$ is a multiple of $\lambda^*$}, \\
0 & \text{otherwise}. \\
\end{cases}
$$
Again, it turns out that the corresponding invariant Hilbert scheme is the affine line for 
certain dominant weights $\lambda$, and is trivial (i.e., consists of the reduced point $Z$)
for all other weights. This result is due to Jansou (see \cite[Th\'eor\`eme 1.1]{Ja07}), 
who also constructed the universal family in the non-trivial cases, as follows.

Assume that the $G$-module $V(\lambda) \oplus V(0) \cong V(\lambda) \times \bA^1$
carries a linear action of a connected reductive group $\tG \supset G$. Assume moreover
that this $\tG$-module is simple, say $V(\tlambda)$, and that the associated cone of highest vectors 
$\tZ \subset V(\tlambda)$ satisfies $Z = \tZ \cap V(\lambda)$ as schemes. 
Then the projection $p : \tZ \to \bA^1$ is a flat family of $G$-subschemes of $V(\lambda)$ 
with fiber $Z$ at $0$, and hence has Hilbert function $h_{\lambda}$. By \cite[Section 2.2]{Ja07}, 
this is in fact the universal family; moreover, all non-trivial cases are obtained from this construction.

One easily shows that the projectivization $\tY : = \tG \cdot [v]$, the closed $\tG$-orbit in 
$\bP\big( V(\tlambda) \big)$, consists of two $G$-orbits: the closed orbit $Y := G \cdot [v]$, 
a hyperplane section of $\tY$, and its (open affine) complement. Moreover, the projective data 
$Y \subset \tY$ uniquely determine the affine data $Z \subset \tZ \subset V(\tlambda)$, 
since the space of global sections of the ample divisor $Y$ on $\tY$ is $V(\tlambda)^*$. 

In fact, the non-trivial cases correspond bijectively to the smooth projective varieties 
where a connected algebraic group acts with two orbits, the closed one being an ample divisor 
(see [loc.~cit., Section 2], based on Akhiezer's classification of certain varieties 
with two orbits in \cite{Ak83}).

Returning to the case that $G = \SL_2$, the universal family for $n = 2$ is obtained by taking
$\tG = \SL_2 \times \SL_2$ where $\SL_2$ is embedded as the diagonal. Moreover,
$$
V(\tlambda) = V(1,1) = V(1) \otimes_k V(1) \cong V(2) \oplus V(0)
$$
where the latter isomorphism is as $\SL_2$-modules; also, $Y = \bP^1$ is the diagonal in 
$\tY = \bP^1 \times \bP^1$.

For $n = 4$, one replaces $\SL_2$ with its quotient $\PSL_2 = \PGL_2$ (that we will keep
denoting by $G$ for simplicity), and takes $\tG = \SL_3$ where $G$ is embedded via its representation
in the $3$-dimensional space $V(2)$. Moreover, $V(\tlambda)$ is the symmetric square
of the standard representation $k^3$ of $\tG$, so that
$$
V(\tlambda) \cong \Sym^2 \big( V(2) \big) \cong V(4) \oplus V(0)
$$
as $G$-modules. Here $Y = \bP^1$ is embedded in $\tY \cong \bP^2$ as a conic.

\medskip

\noindent
(iv) As another generalization of (ii) above, take again $G = \SL_2$ and $X = V(n)$. Assume that $n = 2m$ 
is even and consider the function $h = h_4$ if $m$ is even, and $h = h_2$ if $m$ is odd. 
The invariant Hilbert scheme $\Hilb^G_h(X)$ has been studied in detail by Budmiger in his thesis \cite{Bu10}. 
A closed point of that scheme is the (closed) orbit $G \cdot x^m y^m$, which in fact lies in an 
irreducible component whose underlying reduced scheme is isomorphic to $\bA^1$. But $\Hilb^G_h(X)$ turns out 
to be non-reduced for $m = 6$, and reducible for $m = 8$; see \cite[Section III.1]{Bu10}.
\end{examples}

\section{Basic properties}
\label{sec:bp}

\subsection{Existence}
\label{subsec:ex}

In this subsection, we show how to deduce the existence of the invariant Hilbert
scheme (Theorem \ref{thm:hilb}) from that of the multigraded Hilbert scheme,
proved in \cite{HS04}. We begin with three intermediate results which are of some
independent interest. The first one will allow us to enlarge the acting group $G$. 
To state it, we need some preliminaries on \emph{induced schemes}.

Consider an inclusion of reductive groups $G \subset \tG$; then the homogeneous space
$\tG/G$ is an affine $\tG$-variety equipped with a base point, the image of $e_{\tG}$.
Let $X$ be an affine $G$-scheme of finite type. Then there exists an affine $\tG$-scheme 
of finite type $\tX$, equipped with a $\tG$-morphism 
$$
f: \tX \longrightarrow \tG/G
$$ 
such that the fiber of $f$ at the base point is isomorphic to $X$ as a $G$-scheme. 
Moreover, $\tX$ is the quotient of $\tG \times X$ 
by the action of $G$ via 
$$
g \cdot (\tg, x) := (\tg g^{-1}, g \cdot x)
$$ 
and this identifies $f$ with the morphism obtained from the projection $\tG \times X \to \tG$.
(These assertions follow e.g. from descent theory; see \cite[Proposition 7.1]{MFK94} 
for a more general result).

The scheme $\tX$ satisfies the following property, analogous to \emph{Frobenius reciprocity}
relating induction and restriction in representation theory:  For any $\tG$-scheme $\tY$, 
we have an isomorphism
$$
\Mor^{\tG}(\tX, \tY) \cong \Mor^G(X, \tY)
$$
that assigns to any $f: \tX \to \tY$ its restriction to $X$. The inverse isomorphism assigns to
any $\varphi: X  \to \tY$ the morphism 
$\tG \times X \to \tY$, $(\tg,x) \mapsto \varphi(\tg \cdot x)$
which is $G$-invariant and hence descends to a morphism $\tX \to \tY$.
Thus, $\tX$ is called an {\bf induced scheme}; we denote it by $\tG \times^G X$.

Taking for $\tY$ a $\tG$-module $\tV$, we obtain an isomorphism
\begin{equation}\label{eqn:frob}
\Hom^{\tG}\big( \tV, \cO(\tX) \big) \cong \Hom^G\big( \tV, \cO(X) \big).
\end{equation}
Also, we have isomorphisms of $\tG$-modules
$$
\cO(\tX) \cong \cO(\tG \times X)^G \cong \big( \cO(\tG) \otimes_k \cO(X) \big)^G \cong
\Ind_G^{\tG} \big( \cO(X) \big)
$$
where $\Ind_G^{\tG}$ denotes the induction functor from $G$-modules to $\tG$-modules.

We may now state our first reduction result:

\begin{lemma}\label{lem:ind}
Let $G \subset \tG$ be an inclusion of reductive groups, $X$ an affine $G$-scheme
of finite type, and $\tX := \tG \times^G X$. 
Let $\cZ \subset X \times S$ be a flat family of closed $G$-stable subschemes with 
Hilbert function $h$. Then 
$$
\tcZ := \tG \times ^G \cZ \subset \tX \times S
$$ 
is a flat family of closed $\tG$-stable subschemes, having a Hilbert function $\th$ 
that depends only on $h$. Moreover, if $Hilb^{\tG}_{\th}(\tX)$ is represented by a scheme 
$\cH$, then $Hilb^G_h(X)$ is represented by a union of connected components of $\cH$.
\end{lemma}

\begin{proof}
Consider a simple $\tG$-module $\tM$ and the associated sheaf of covariants 
$$
\cF_{\tM} = \Hom^{\tG}(\tM, \tp_*\cO_{\tG \times^G \cZ})
$$
where $\tp : \tG \times^G \cZ \to S$ denotes the projection. By using (\ref{eqn:frob}), this yields
$$
\cF_{\tM} \cong \Hom^G(\tM, p_*\cO_{\cZ}).
$$
Thus, the sheaf of $\cO_S$-modules $\cF_{\tM}$ is locally free of rank 
$$
\sum_{M \in \Irr(G)} \dim \Hom^G(M,\tM) \; h(M) =: \th(\tM).
$$
It follows that $\tG \times ^G \cZ$ is flat with Hilbert function $\th$ just defined.
This shows the first assertion, and defines a morphism of functors  
$Hilb^G_h(X) \to Hilb^{\tG}_{\th}(\tX)$, $\cZ \mapsto \tcZ := \tG \times^G \cZ$. 

Next, consider a family of closed $\tG$-stable subschemes $\tcZ \subset \tX \times S$. 
Then the composite morphism 
$$
\CD
\tcZ @>{\tq}>> \tX @>{f}>> \tG/G
\endCD
$$ 
is $\tG$-equivariant. It follows that $\tcZ = \tG \times^G \cZ$ 
for some $G$-stable subscheme $\cZ \subset X \times S$. If $\tcZ$ is flat over $S$,
then by the preceding step, the sheaf of $\cO_S$-modules $\Hom^G(\tM, p_*\cO_{\cZ})$
is locally free for any simple $\tG$-module $\tM$. But every simple $G$-module $M$
is a submodule of some simple $\tG$-module $\tM$ (indeed, $M$ is a quotient of $\Ind_G^{\tG}(M)$,
and hence a quotient of a simple $\tG$-submodule). It follows that the sheaf of 
$\cO_S$-modules $\Hom^G(M, p_*\cO_{\cZ})$ is a direct factor of $\Hom^G(\tM, p_*\cO_{\cZ})$
and hence is locally free of finite rank. Thus, $\cZ$ is flat and multiplicity-finite over $S$;
hence $\cZ$ has a Hilbert function $h'$ such that $\widetilde{h'} = \th$, if $S$ is connected.
When $h' = h$, the assignements $\cZ \mapsto \tcZ$ and $\tcZ \mapsto \cZ$ are mutually inverse. 
Taking for $S$ a connected component of $\cH$ and for $\tcZ$ the pull-back of the
universal family, we obtain the final assertion.
\end{proof}

By the preceding lemma, we may replace $G$ with $\GL_n$; in particular, 
we may assume that $G$ is \emph{connected}. Our second result, a variant of
\cite[Theorem 1.7]{AB05}, will allow us to replace $G$ with a torus.
As in Subsection \ref{subsec:rg}, we choose a Borel subgroup $B \subset G$
with unipotent part $U$, and a maximal torus $T \subset B$. We consider an affine
$G$-scheme $X = \Spec(A)$ and a function $h : \Lambda^+ \to \bN$; we extend $h$ to
a function $\bar{h}: \Lambda \to \bN$ with values $0$ outside $\Lambda^+$, as
in (\ref{eqn:bh}).

\begin{lemma}\label{lem:Uinv}
With the preceding notation, assume that $Hilb^T_{\bar{h}}(X/\!/U)$ is represented 
by a scheme $\cH$. Then $Hilb^G_h(X)$ is represented by a closed subscheme of $\cH$. 
\end{lemma}

\begin{proof}
We closely follow the argument of \cite[Theorem 1.7]{AB05}. Given a scheme $S$ and 
a flat family $\cZ \subset X \times S$ of closed $G$-stable subschemes with Hilbert 
function $h$, we obtain a family $\cZ/\!/U \subset X/\!/U \times S$ of closed
$T$-stable subschemes which is again flat and has Hilbert function $\bar{h}$, by 
Remark \ref{rem:fam}(iii). Observe that $\cZ/\!/U$ uniquely determines $\cZ$: 
indeed, $\cZ$ corresponds to a $G$-stable sheaf of ideals
$$
\cI \subset A \otimes_k \cO_S
$$
such that each quotient $(A \otimes_k \cO_S)^U_{\lambda}/ \cI^U_{\lambda}$
is locally free of rank $\bar{h}(\lambda)$. Moreover, $\cZ/\!/U$ corresponds
to the $T$-stable sheaf of ideals 
$$
\cI^U \subset A^U \otimes_k \cO_S
$$ 
which generates $\cI$ as a sheaf of $\cO_S$-$G$-modules.

We now express the condition for a given $T$-stable sheaf of ideals
$$
\cJ \subset A^U \otimes_k \cO_S
$$
such that each quotient $(A \otimes_k \cO_S)^U_{\lambda}/ \cJ_{\lambda}$ is locally free of rank 
$\bar{h}(\lambda)$, to equal $\cI^U$ for some $G$-stable sheaf of ideals $\cI$ as above. 
This is equivalent to the condition that the $\cO_S$-$G$-module
$$
\cI := \langle G \cdot \cJ \rangle \subset  A \otimes_k \cO_S
$$ 
generated by $\cJ$, is a sheaf of ideals, i.e., $\cI$ is stable under
multiplication by $A$. By highest weight theory, this means that
$$
(\cI \cdot  A)^U  \subset \cJ.
$$ 
We will translate the latter condition into the vanishing of certain morphisms of
locally free sheaves over $S$, arising from the universal family of $\cH$ 
via the classifying morphism 
$$
f: S \longrightarrow \cH.
$$ 
For this, consider three dominant weights $\lambda$, $\mu$, $\nu$ and a 
copy of the simple $G$-module $V(\nu)$ in $V(\lambda) \otimes_k V(\mu)$, 
with highest weight vector 
$$
v \in \big( V(\lambda) \otimes_k V(\mu) \big)^U_{\nu}.
$$ 
We may write 
$$
v = \sum_i c_i (g_i \cdot v_{\lambda}) \otimes (h_i \cdot v_{\mu}),
$$
a finite sum where $c_i \in k$ and $g_i, h_i \in G$. This defines a morphism
of sheaves of $\cO_S$-modules
$$
A^U_{\lambda} \otimes_k \cJ_{\mu} \longrightarrow A^U_{\nu} \otimes_k \cO_S,
\quad
a \otimes b \longmapsto \sum_i c_i (g_i \cdot a)(h_i \cdot b).
$$
Composing with the quotient by $\cJ_{\nu}$ yields a morphism of sheaves 
of $\cO_S$-modules,
$$
F_v : A^U_{\lambda} \otimes_k \cJ_{\mu} \longrightarrow 
(A^U_{\nu} \otimes_k \cO_S)/\cJ_{\nu}.
$$
Our condition is the vanishing of these morphisms $F_v$ for all triples
$(\lambda,\mu,\nu)$ and all $v$ as above. Now $(A^U_{\nu} \otimes_k \cO_S)/\cJ_{\nu}$
and $\cJ_{\mu}$ are the pull-backs under $f$ of the analogous locally 
free sheaves on $\cH$. This shows that the Hilbert functor $Hilb^G_h(X)$ is represented 
by the closed subscheme of $\cH$ obtained as the intersection of the zero loci of the $F_v$. 
\end{proof}

Our final reduction step will allow us to enlarge $X$:

\begin{lemma}\label{lem:sub}
Let $X$ be a closed $G$-subscheme of an affine $G$-scheme $Y$ of finite type. 
If $Hilb^G_h(Y)$ is represented by a scheme $\cH$, then $Hilb^G_h(X)$ is represented
by a closed subscheme of $\cH$.
\end{lemma}

\begin{proof}
Let $X = \Spec(A)$ and $Y = \Spec(B)$, so that we have an exact sequence
$$
0 \longrightarrow I \longrightarrow B \longrightarrow A \longrightarrow 0
$$
where $I$ is a $G$-stable ideal of $B$. For any $M \in \Irr(G)$, this yields an
exact sequence for modules of covariants over $B^G$:
$$
0 \longrightarrow I_M \longrightarrow B_M \longrightarrow A_M \longrightarrow 0.
$$
Next, consider a scheme $S$ and a flat family $p: \cZ \to S$ of closed $G$-stable 
subschemes of $Y$, with Hilbert function $h$. Then each associated sheaf of covariants $\cF_M$ 
is a locally free quotient of $B_M \otimes_k \cO_S$, of rank $h(M)$; this defines 
a linear map $q_M : B_M \to H^0(S,\cF_M)$. Moreover, $\cZ$ is contained in $X \times S$
if and only if the image of $I \otimes_k \cO_S$ in $p_* \cO_{\cZ}$ is zero; equivalently, $q_M(I_M) = 0$ 
for all $M \in \Irr(G)$. Taking for $p$ the universal family of $\cH$, it follows that 
the invariant Hilbert functor $Hilb^G_h(X)$ is represented by the closed subscheme of $\cH$, 
intersection of the zero loci of the subspaces 
$q_M(I_M) \subset \Gamma \big( \Hilb_h^G(Y),\cF_M \big)$ for all $M \in \Irr(G)$.
\end{proof}

Summarizing, we may reduce to the case that $G$ is a maximal torus of $\GL_n$ by 
combining Lemmas \ref{lem:ind} and \ref{lem:Uinv}, and then to the case that $X$ is 
a finite-dimensional $G$-module by Lemma \ref{lem:sub}. Then the invariant Hilbert
scheme is exactly the multigraded one, as noted in Remark \ref{rem:tuf}(iii). 

\begin{remarks}\label{rem:ex}
(i) The proof of Lemma \ref{lem:Uinv} actually shows that the invariant Hilbert functor 
$Hilb^G_h(X)$ is a closed subfunctor of $Hilb^T_h(X/\!/U)$. Likewise, in the 
setting of Lemma \ref{lem:ind} (resp. of Lemma \ref{lem:sub}), $Hilb^G_h(X)$ is a closed 
subfunctor of $Hilb^{\tG}_{\th}(\tX)$ (resp. of $Hilb^G_h(Y)$).

\medskip

\noindent
(ii) The arguments of this subsection establish the existence of the invariant Hilbert scheme 
over any field of characteristic $0$. Indeed, highest weight theory holds for $\GL_n$ in
that setting, whereas it fails for non-split reductive groups.
\end{remarks}

\subsection{Zariski tangent space}
\label{subsec:zts}

In this subsection, we consider a reductive group $G$, an affine $G$-scheme of finite type
$X = \Spec(A)$, and a function $h : \Irr(G) \to \bN$. We study the Zariski tangent space 
$T_Z \Hilb^G_h(X)$ to the invariant Hilbert scheme at an arbitrary closed point $Z$, i.e., 
at a closed $G$-stable subscheme of $X$ with Hilbert function $h$. As a first step, we obtain:

\begin{proposition}\label{prop:zts}
With the preceding notation, we have
\begin{equation}\label{eqn:zts}
T_Z \Hilb_h^G(X) = \Hom^G_A(I, A/I) 
= \Hom^G_{\cO(Z)}\big( I/I^2,\cO(Z) \big)
\end{equation}
where $I\subset A$ denotes the ideal of $Z$, and $\Hom^G_A$ stands for the space of $A$-linear, 
$G$-equivariant maps.
\end{proposition}

Indeed, $\Hom_A(I, A/I)$ parametrizes the {\bf first-order deformations of $Z$ in $X$}, 
i.e., those closed subschemes 
$$
\cZ \subset X \times \Spec\, k[\varepsilon]
$$ 
(where $\varepsilon^2 = 0$) that are flat over $\Spec \, k[\varepsilon]$ and satisfy $\cZ_s = Z$ 
where $s$ denotes the closed point of $\Spec \, k[\varepsilon]$; see e.g. \cite[Section 3.2]{Se06}). 
The subspace $\Hom_A^G(I, A/I)$ parametrizes the $G$-stable deformations.

\begin{example}\label{ex:m2r}
Let $G = \SL_2$ and $X := r V(2)$ (the direct sum of $r$ copies of $V(2)$), where $r$ is a positive integer.
We consider the invariant Hilbert scheme $\Hilb^G_h(X)$, where $h = h_2$ is as defined in 
(\ref{eqn:hn}), and show that the Zariski tangent space at any closed point $Z$ has dimension $r$. 

Indeed, the $r$ projections $p_1,\ldots,p_r : Z \to V(2)$ are all proportional, since $h_2(2) = 1$. 
Thus, we may assume that $Z$ is contained in the first copy of $V(2)$, for an appropriate choice of projections.
Then the condition that $h_2(0) =1$ implies that $Z$ is contained in the scheme-theoretic fiber 
of the discriminant $\Delta$ at some scalar $t$. Since that fiber has also Hilbert function $h_2$, 
we see that equality holds: the ideal of $Z$ satisfies
$$
I = \big( \Delta(p_1) - t, p_2, \cdots, p_r \big).
$$
In particular, $Z$ is a complete intersection in $X$, and the $\cO(Z)$-$G$-module $I/I^2$
is freely generated by the images of $\Delta(p_1) - t, p_2, \ldots, p_r$.
This yields an isomorphism of $\cO(Z)$-$G$-modules
$$
I/I^2 \cong \cO(Z) \otimes_k \big( V(0) \oplus (r-1) V(2) \big).
$$
As a consequence,
$$
\Hom_{\cO(Z)}^G\big( I/I^2, \cO(Z) \big) \cong 
\Hom^G\big( V(0) \oplus (r-1) V(2), \cO(Z) \big)
$$
has dimension $h_2(0) + (r - 1) h_2(2) = r$. Together with (\ref{eqn:zts}), this implies the statement.

In fact, $\Hilb^G_h(X)$ is a smooth irreducible variety of dimension $r$, as we will see in 
Example \ref{ex:aut}(iv). Specifically,  $\Hilb^G_h(X)$ is the total space of the line bundle of degree 
$-2$ on $\bP^{r-1}$, see Example \ref{ex:qsm}(iv).
\end{example}

The isomorphism (\ref{eqn:zts}) is the starting point of a local analysis of the 
invariant Hilbert scheme, in relation to deformation theory (for the latter, see \cite{Se06}). 
We will present a basic and very useful result in that direction; it relies on the following:

\begin{lemma}\label{lem:prof}
Let $\cM$ be a coherent sheaf on an affine scheme $Z$, and $M = H^0(Z,\cM)$ the associated 
finitely generated module over $R := \cO(Z)$. Let $Z_0 \subset Z$ be a dense open subscheme 
and denote by $\iota : Z_0 \to Z$ the inclusion map. Then the pull-back
$$
\iota^* : \Hom_R(M,R) = \Hom_Z(\cM,\cO_Z) \longrightarrow 
\Hom_{Z_0}(\cM_{\vert Z_0},\cO_{Z_0})
$$
is injective. 

If $Z$ is a normal irreducible variety and the complement $Z \setminus Z_0$ has codimension $\geq 2$ 
in $Z$, then $\iota^*$ is an isomorphism.
\end{lemma}

\begin{proof}
Choose a presentation of the $R$-module $M$,
$$
\CD
R^m @>{A}>> R^n @>>> M @>>> 0,
\endCD
$$
where $A$ is a matrix with entries in $R$. This yields an exact sequence of $R$-modules
$$
\CD
0 @>>> \Hom_R(M,R) @>>> R^n @>{B}>> R^m
\endCD
$$ 
where $B$ denotes the transpose of $A$. In other words, 
$\Hom_R(M,R)$ consists of those $(f_1,\ldots,f_n) \in R^n$
that are killed by $B$. This implies both assertions, 
since $\iota^* : R = \cO(Z) \to \cO(Z_0)$ is injective,
and is an isomorphism under the additional assumptions.
\end{proof}

We may now obtain a more concrete description of the Zariski tangent space at 
a $G$-orbit closure:

\begin{proposition}\label{prop:orb}
Let $G$ be a reductive group, $V$ a finite-dimensional $G$-module, and $v$ a point of $V$.
Denote by $Z \subset V$ the closure of the orbit $G \cdot v$ and by $h$ the Hilbert function 
of $\cO(Z)$. Let $G_v \subset G$ be the isotropy group of $v$, and $\fg$ the Lie algebra of $G$. 
Then  
\begin{equation}\label{eqn:orb}
T_Z \Hilb^G_h(V) \hookrightarrow \big(V/ \fg \cdot v \big)^{G_v}
\end{equation}
where $G_v$ acts on $V/ \fg \cdot v$ via its linear action on $V$ which 
stabilizes the subspace $\fg \cdot v = T_v(G \cdot v)$. 

Moreover, equality holds in (\ref{eqn:orb}) if $Z$ is normal and the boundary 
$Z \setminus G \cdot v$ has codimension $\geq 2$ in $Z$.
\end{proposition}

\begin{proof}
We apply Proposition \ref{prop:zts} and Lemma \ref{lem:prof} by taking 
$M = I/I^2$ and $Z_0 = G \cdot v$. This yields an injection of $T_Z \Hilb^G_h(V)$ 
into
$$
W := \Hom^G_{Z_0}\big( (\cI/\cI^2)_{\vert Z_0}, \cO_{Z_0} \big),
$$
where $\cI$ denotes the ideal sheaf of $Z$ in $V$. Moreover, $T_Z \Hilb^G_h(V) = W$
under the additional assumptions. Since $Z_0$ is a smooth subvariety of the smooth
irreducible variety $V$, the conormal sheaf $(\cI/\cI^2)_{\vert Z_0}$ is locally free. 
Denoting the dual (normal) sheaf by $\cN_{Z_0/V}$, we have 
$$
W = H^0(Z_0,\cN_{Z_0/V})^G.
$$ 
But $\cN_{Z_0/V}$ is the sheaf of local sections of the normal bundle, and 
the total space of that bundle is the $G$-variety $G \times^{G_v} \cN_{Z_0/V,v}$
equipped with the projection to $G/G_v = G \cdot v$. Moreover, we have isomorphisms
of $G_v$-modules
$$
\cN_{Z_0/V, v} \cong T_v(V)/T_v(G \cdot v) \cong V/\fg \cdot v.
$$ 
It follows that 
$$
H^0(Z_0,\cN_{Z_0/V})^G \cong \big( \cO(G/G_v) \otimes_k \cN_{Z_0/V,v} \big)^G
\cong \Hom^G(G/G_v,\cN_{Z_0/V, v}) \cong \cN_{Z_0/V, v}^{G_v}.
$$
This implies our assertions.
\end{proof}

We refer to \cite[Section 1.4]{AB05} for further developments along these lines, 
including a relation to (non-embedded) first-order deformations, 
and to \cite[Section 3]{PvS10} for generalizations where the boundary may have irreducible components
of codimension $1$. The obstruction space for $G$-invariant deformations is considered in 
\cite[Section 3.5]{Cu09}. 

\begin{examples}\label{ex:zts}
(i) Let $\ulambda = (\lambda_1,\ldots,\lambda_N)$ be a list of pairwise distinct weights of a torus $T$, 
and $\Hilb^T(\ulambda)$ the associated \emph{toric Hilbert scheme} as in Example \ref{ex:tuf}(i). 
Let $Z = \overline{T \cdot v}$ where $v$ is a general point of $V= \bA^N$ 
(i.e., all of its coordinates are non-zero). Then the stabilizer $T_v$ is the kernel 
of the homomorphism
\begin{equation}\label{eqn:ulambda}
\ulambda : T \longrightarrow (\bG_m)^N, \quad 
t \longmapsto \big( \lambda_1(t),\ldots,\lambda_N(t) \big)
\end{equation}
and hence acts trivially on $V$. Thus, the preceding proposition just yields an inclusion
$$
\iota : T_Z \Hilb^T(\ulambda) \hookrightarrow V/ \ft \cdot v
$$
where $\ft$ denotes the Lie algebra of $T$.

In fact, \emph{$\iota$ is an isomorphism}. Indeed, moving $v$ among the general points defines 
a family $p : \cZ \to (\bG_m)^N$ of $T$-orbit closures in $V$, and hence a morphism 
$f: (\bG_m)^N \to \Hilb^T(\ulambda)$. Moreover, the differential of $f$
at $v$ composed with $\iota$ yields the quotient map $V \to V/\ft \cdot v$; hence 
$\iota$ is surjective. (See Example \ref{ex:aut}(i) for another version of this argument, 
based on the natural action of $(\bG_m)^N$ on $\Hilb^T(\ulambda)$.)
 
\medskip

\noindent

(ii) As in Example \ref{ex:tuf}(ii), let $G = \SL_2$, $X = V(n)$ and $Z$ the variety of $n$-th powers
of linear maps. Then $Z = \overline{G \cdot v} = G \cdot v \cup \{0\}$, where $v := y^n$ is 
a highest weight vector; moreover, $Z$ is a normal surface. Thus, we may apply the preceding 
proposition to determine $T_Z \Hilb^G_h\big( V(n) \big)$, where $h = h_n$ is the function
(\ref{eqn:hn}).

The stabilizer $G_{y^n}$ is the semi-direct product of the additive group $U$ (acting via
$x \mapsto x + t y$, $y \mapsto y$) with the group $\mu_n$ of $n$-th roots of unity acting via 
$x \mapsto \zeta x$, $y \mapsto \zeta^{-1} y$. Also, $\fg \cdot v$ is spanned by the monomials
$y^n$ and $xy^{n-1}$, and $V/\fg \cdot v$ has basis the images of the remaining monomials 
$x^n, x^{n-1}y, \ldots, x^2y^{n-2}$. It follows that $(V/\fg \cdot v)^U$ is spanned by the image of 
$x^2 y^{n-2}$; the latter is fixed by $\mu_n$ if and only if $n = 2$ or $n = 4$. We thus obtain: 
$$
T_Z \Hilb^G_h(V) = 
\begin{cases}
k & \text{if $n = 2$ or $4$}, \\
0 & \text{otherwise}. \\
\end{cases}
$$

\medskip

\noindent
(iii) More generally, let $G$ be an arbitrary connected reductive group, $V= V(\lambda)$ 
a simple $G$-module of dimension $\geq 2$, $v = v_{\lambda}$ a highest weight vector,
and $Z = \overline{G \cdot v}$ as in Example \ref{ex:tuf}(iii). Then the stabilizer
of the highest weight line $[v]$ is a parabolic subgroup $P \supset B$, and the
character $\lambda$ of $B$ extends to $P$; moreover, $G_v$ is the kernel of that extended
character. Also, $Z$ is normal and its boundary (the origin) has codimension $\geq 2$. 
Thus, Proposition \ref{prop:orb} still applies to this situation. Combined with
arguments of combinatorial representation theory, it yields that $T_Z \Hilb^G_h(V) = 0$
unless $\lambda$ belongs to an explicit (and rather small) list of dominant weights; 
in that case, $T_Z \Hilb^G_h(V) = k$ (see \cite[Section 1.3]{Ja07}).
\end{examples}

\subsection{Action of equivariant automorphism groups}
\label{subsec:aeag}

As in the previous subsection, we fix an affine $G$-scheme of finite type $X$ and a function
$h: \Irr(G) \to \bN$. We obtain a natural equivariance property of the corresponding 
invariant Hilbert scheme.

\begin{proposition}\label{prop:aut}
Let $H$ be an algebraic group, and $\beta : H \times X \to X$ an action by $G$-automorphisms. 
Then $\beta$ induces an $H$-action on $\Hilb^G_h(X)$ that stabilizes the universal family 
$\Univ^G_h(X) \subset X \times \Hilb^G_h(X)$.
\end{proposition}

\begin{proof}
Given a flat family of $G$-stable subschemes $\cZ \subset X \times S$ with Hilbert function $h$, 
we construct a flat family of $G$-stable subschemes 
$$
\cW \subset H \times X \times S
$$ 
with the same Hilbert function, where $G$ acts on $H \times X \times S$ via 
$g \cdot (h,x,s) = (h, g \cdot x, s)$. For this, form the cartesian square
\begin{equation}\label{eqn:act}
\CD
\cW @>>> H \times X \times S \\
@VVV @V{\beta \times \id_S}VV \\ 
\cZ @>>> X \times S. \\
\endCD
\end{equation}
Then $\cW$ is a closed subscheme of $H \times X \times S$, stable under the given $G$-action
since $\beta$ is $G$-equivariant. Moreover, the map 
$$
H \times X \longrightarrow H \times X, \quad (h,x) \longmapsto (h, h \cdot x) 
$$
is an isomorphism; thus, the square
\begin{equation}\label{eqn:com}
\CD
H \times X \times S @>>> H \times S \\
@V{\beta \times \id_S}VV @VVV \\
X \times S @>>> S \\
\endCD
\end{equation}
(where the non-labeled arrows are the projections) is cartesian. By composing 
(\ref{eqn:act}) and (\ref{eqn:com}), we obtain a cartesian square 
$$
\CD
\cW @>>> H \times S \\
@VVV @VVV \\
\cZ @>>> S \\
\endCD
$$
i.e., an isomorphism
$$
\cW \cong (H \times S) \times_S \cZ \cong H \times \cZ
$$
where the morphism $H \times \cZ \to \cW$ is given by $(h,z) \mapsto \beta(h^{-1}, z)$.  
It follows that $\cW$ is flat over $H \times S$ with Hilbert function $h$.

Applying this construction to $\cZ = \Univ^G_h(X)$ and $S = \Hilb^G_h(X)$ yields a flat family 
$\cW$ with Hilbert function $h$ and base $H \times \Hilb^G_h(X)$, and hence a morphism of schemes
$$
\varphi : H \times \Hilb^G_h(X) \longrightarrow \Hilb^G_h(X)
$$
such that $\cW$ is the pull-back of the universal family. Since the composite morphism
$$
\CD 
X @>{e_H \times \id_X}>> H \times X @>{\beta}>> X
\endCD
$$
is the identity, it follows that the same holds for the composite
$$
\CD 
\Hilb^G_h(X) @>{e_H \times \id_{\Hilb^G_h(X)}}>> H \times \Hilb^G_h(X) @>{\varphi}>> \Hilb^G_h(X).
\endCD
$$
Likewise, $\varphi$ satisfies the associativity property (\ref{eqn:asso}).
Thus, $\varphi$ is an $H$-action on $\Hilb^G_h(X)$. 

To show that the induced $H$-action on $X \times \Hilb^G_h(X)$ stabilizes the closed
subscheme $\Univ^G_h(X)$, note that $\cW \subset H \times X \times \Hilb^G_h(X) $ is the closed
subscheme defined by $\big( (\beta(h^{-1}, x \big),s) \in \Univ^G_h(X)$ with an obvious notation 
(as follows from the cartesian square (\ref{eqn:act})). But $\cW$ is also defined by 
$\big( x, \varphi(h,s) \big) \in \Univ^G_h(X)$ (since it is the pull-back of $\Univ^G_h(X)$).
This yields the desired statement.
\end{proof}
 
In the case that $X$ is a finite-dimensional $G$-module, say $V$, a natural choice for $H$ is 
the automorphism group of the $G$-module $V$, i.e., the centralizer of $G$ in $\GL(V)$; 
we denote that group by $\GL(V)^G$. To describe it, consider the isotypical decomposition
\begin{equation}\label{eqn:is}
V \cong m_1 V(\lambda_1) \oplus \cdots \oplus m_r V(\lambda_r),
\end{equation}
where $\lambda_1, \ldots, \lambda_r$ are pairwise distinct dominant weights, and the multiplicities
$m_1, \ldots, m_r$ are positive integers. By Schur's lemma, $\GL(V)^G$ preserves each
isotypical component $m_i V(\lambda_i) \cong k^{m_i} \otimes_k V(\lambda_i)$ and acts there 
via a homomorphism to $\GL_{m_i}$. It follows that
$$
H \cong \GL_{m_1} \times \cdots \times \GL_{m_r}.
$$
In particular, the center of $\GL(V)^G$ is a torus $(\bG_m)^r$ acting on $V$ via
$$
(t_1,\ldots,t_r) \cdot (v_1, \ldots,v_r) = (t_1 v_1, \ldots, t_r v_r)
\quad \text{where} \quad v_i \in m_i V(\lambda_i).
$$ 

\begin{examples}\label{ex:aut}
(i) For a torus $T$ acting linearly on $V = \bA^N$ via pairwise distinct weights 
$\lambda_1,\ldots,\lambda_N$, the group $\GL(V)^T$ is just the diagonal torus
$(\bG_m)^N$. In particular, this yields an action of $(\bG_m)^N$ on the toric
Hilbert scheme $\Hilb^T(\ulambda)$ where $\ulambda =(\lambda_1,\ldots,\lambda_N)$. 
The stabilizer of a general orbit closure $Z = \overline{T \cdot v}$ is the image of the homomorphism
$\ulambda : T \to (\bG_m)^N$ (\ref{eqn:ulambda}) with kernel $T_v$. Thus, the orbit $(\bG_m)^N \cdot Z$ 
is a smooth subvariety of $\Hilb^T(\ulambda)$, of dimension 
$$
N - \dim \ulambda(T) = N - \dim(T) + \dim(T_v) = \dim(V/\ft \cdot v).
$$
Since $\dim T_Z \Hilb^T(\ulambda) \leq \dim(V/\ft \cdot v)$
by Example \ref{ex:zts}(i), it follows that equality holds. We conclude that 
\emph{the orbit $(\bG_m)^N \cdot Z$ is open in $\Hilb^T(\ulambda)$}.

As a consequence, the closure of that orbit is an irreducible component of the toric Hilbert scheme, 
equipped with its reduced subscheme structure: the {\bf main component}, also called the 
{\bf coherent component}. Its points are the general $T$-orbit closures and their flat limits 
as closed subschemes of $\bA^N$. The normalization of the main component is a quasi-projective toric 
variety under the quotient torus $(\bG_m)^N/\ulambda(T)$. 

In particular, the main component is a point if and only if the homomorphism $\ulambda$
is surjective, i.e., the weights $\lambda_1,\ldots,\lambda_N$ are linearly independent. 
In that case, one easily sees that the whole toric Hilbert scheme is a (reduced) point. 

\medskip

\noindent
(ii) The automorphism group of the simple $\SL_2$-module $V(n)$ is just $\bG_m$ 
acting by scalar multiplication. For the induced action on the invariant Hilbert scheme, 
a closed point is fixed if and only if its ideal is homogeneous. If the Hilbert function is 
the function $h_n$ defined in (\ref{eqn:hn}), then there is a unique such ideal $I$: 
that of the variety of $n$-th powers. Indeed, we have an isomorphism of $\SL_2$-modules
$$
\cO\big(V(n)\big)/I \cong \Sym \big(V(n) \big)/I \cong \bigoplus_{m=0}^{\infty} V(mn).
$$
Moreover, the graded $G$-algebra $\Sym \big(V(n) \big)/I$ is generated by (the image of) $V(n)$,
its component of degree $1$. By an easy induction on $m$, it follows that the component 
of an arbitrary degree $m$ of that algebra is isomorphic to $V(mn)$. But the $G$-submodule 
$V(mn) \subset \Sym^m \big( V(n) \big)$ 
has a unique $G$-stable complement: the sum of all other simple submodules. This determines each
homogeneous component of the graded ideal $I$.

\medskip

\noindent
(iii) The preceding argument adapts to show that the cone of highest weight vectors is the unique 
fixed point for the $\bG_m$-action on $\Hilb_{h_{\lambda}}^G\big( V(\lambda) \big)$, with the notation 
of Example \ref{ex:tuf}(iii).

\medskip

\noindent
(iv) As in Example \ref{ex:m2r}, let $G = \SL_2$, $V = r V(2)$ and $h = h_2$. 
Then $H = \GL(V)^G$ is the general linear group $\GL_r$ acting on $V \cong k^r \otimes_k V(2)$ via its standard
action on $k^r$. For the induced action of $H$ on $\Hilb^G_h(V)$, the closed points form two orbits 
$\Omega_1$, $\Omega_0$, with representatives $Z_1$, $Z_0$ defined by 
$\Delta(p_1) = 1$ ($\Delta(p_1) = 0$) and $p_2 = \ldots = p_r =0$.
One easily checks that the isotropy group $H_{Z_0}$ is the parabolic subgroup of $\GL_r$ that stabilizes 
the first coordinate line $ke_1$; as a consequence, $\Omega_0 \cong \bP^{r-1}$. Also, $H_{Z_1}$ is the
stabilizer of $\pm e_1$, and hence $\Omega_1 \cong (\bA^r \setminus \{0\})/\pm 1$ where $\GL_r$ acts
linearly on $\bA^r$. Since the family $\cZ$ of Example \ref{ex:fam}(ii) has general fibers in $\Omega_1$
and special fiber in $\Omega_0$, we see that $\Omega_0$ is contained in the closure of $\Omega_1$. 
As a consequence, \emph{$\Hilb^G_h(V)$ is irreducible of dimension $r$}. Since its Zariski tangent space has 
dimension $r$ at each closed point, it follows that \emph{$\Hilb^G_h(V)$ is a smooth variety}. 
\end{examples}

\subsection{The quotient-scheme map}
\label{subsec:qsm}

We keep the notation of the previous subsection, and consider a family of $G$-stable closed 
subschemes $\cZ \subset X \times S$ over some scheme $S$. Recall that the sheaf of $\cO_S$-algebras 
$\cF_0 = (p_* \cO_{\cZ})^G$ is a quotient of $A^G \otimes_k \cO_S$ where $A = \cO(X)$. 
This defines a family of closed subschemes
$$
\cZ/\!/G \subset X/\!/ G \times S
$$
where we recall that $X/\!/G = \Spec(A^G)$. If $\cZ$ is flat over $S$ with Hilbert function $h$, 
then $\cF_0$ is locally free over $S$, of rank
$$
n := h(0).
$$ 
Thus, $\cZ/\!/G$ defines a morphism to the punctual Hilbert scheme of the quotient,
$$
f: S \longrightarrow \Hilb_n(X/\!/G).
$$
Applying this construction to the universal family yields a morphism
\begin{equation}\label{eqn:qsm}
\gamma : \Hilb^G_h(X) \longrightarrow \Hilb_n(X/\!/G)
\end{equation}
that we call the {\bf quotient-scheme map}.

\begin{proposition}\label{prop:proj}
With the preceding notation, the morphism (\ref{eqn:qsm}) is projective.
In particular, if the scheme $X/\!/G$ is finite, or equivalently, if $X$ 
contains only finitely many closed $G$-orbits, then $\Hilb^G_h(X)$ is projective.
\end{proposition}

\begin{proof}
Since $\Hilb^G_h(X)$ is quasi-projective, it suffices to show that $\gamma$ is proper. 
For this, we use the valuative criterion of properness for schemes of finite type: 
let $R$ be a discrete valuation ring containing $k$ and denote by $K$ the fraction field of $R$. 
Let $\cZ_K \in \Hilb_h^G(X)(K)$ and assume that $\gamma(\cZ_K) \in \Hilb_n(X/\!/G)(K)$ 
admits a lift to $\Hilb_n(X/\!/G)(R)$. Then we have to show that $\cZ_K$ admits a lift to
$\Hilb_h^G(X)(R)$. 

In other words, we have a family 
$$
\cZ_K \subset X \times \Spec(K)
$$
of closed $G$-stable subschemes with Hilbert function $h$, such that the family
$$
\cZ_K/\!/G \subset X/\!/G \times \Spec(K)
$$ 
extends to a family in $X/\!/G \times \Spec(R)$. These data correspond to a $G$-stable ideal
$$
I_K \subset A \otimes_k K
$$ 
such that $I_K^G \subset A^G \otimes_k K$ equals $J \otimes_R K$, where 
$$
J \subset A^G \otimes_k R
$$ 
is an ideal such that the $R$-module $(A^G \otimes_k R)/J$ is free of rank $h(0)$. Then 
$$
J \subset (A^G \otimes_k R) \cap (J \otimes_R K),
$$
where the intersection is considered in $A \otimes_R K$. Moreover, the quotient 
$R$-module $(A^G \otimes_k R) \cap (J \otimes_R K)/J$ is torsion; on the other hand, 
this quotient is contained in the free $R$-module $(A^G \otimes_k R)/J$. Thus, 
$$
J = (A^G \otimes_k R) \cap (J \otimes_R K) = (A^G \otimes_k R) \cap I_K.
$$
We now consider
$$
I := (A  \otimes_k R) \cap I_K.
$$
Clearly, $I$ is an ideal of $A \otimes_k R$ satisfying $I^G = J$ and 
$I \otimes_R K = I_K$. Thus, the $R$-module 
$$
\big( (A \otimes_k R)/I \big)^G = (A^G \otimes_k R)/I^G
$$ 
is free of rank $n$. Moreover, each module of covariants 
$\big( (A \otimes_k R)/I \big)_M$ is finitely generated over 
$\big( (A \otimes_k R)/I \big)^G$, and torsion-free by construction.
Hence the $R$-module $\big( (A \otimes_k R)/I \big)_M$ is free;
tensoring with $K$, we see that its rank is $h(M)$. Thus, $I$ 
corresponds to an $R$-point of $\Hilb_h^G(X)$, which is the desired lift.
\end{proof}

\begin{remarks}\label{rem:hs}
(i) The morphism (\ref{eqn:qsm}) is analogous to the {\bf Hilbert-Chow morphism}, or 
{\bf cycle map}, that associates with any closed subscheme $Z$ of the projective space $\bP^N$, 
the support of $Z$ (with multiplicities) viewed as a point of the Chow variety of $\bP^N$. 

The cycle map may be refined in the setting of punctual Hilbert schemes: given a quasi-projective 
scheme $X$ and a positive integer $n$, there is a natural morphism 
\begin{equation}\label{eqn:cm}
\varphi_n : \Hilb_n(X) \longrightarrow X^{(n)}
\end{equation}
where $X^{(n)}$ denotes the {\bf $n$-th symmetric product of $X$}, i.e., the quotient of the
product $X \times \cdots \times X$ ($n$ factors) by the action of the symmetric group 
$S_n$ that permutes the factors; this is a quasi-projective scheme with closed points the
effective $0$-cycles of degree $n$ on $X$. Moreover, $\varphi_n$ induces the cycle map on closed
points, and is a projective morphism (for these results, see \cite[Theorem 2.16]{Be08}).

In the setting of Proposition \ref{prop:proj}, let 
\begin{equation}\label{eqn:hccm}
\delta : \Hilb^G_h(X) \longrightarrow (X/\!/G)^{(n)}
\end{equation}
denote the composite of (\ref{eqn:qsm}) with (\ref{eqn:cm}). Then $\delta$ is a projective morphism, 
and $(X/\!/G)^{(n)}$ is affine. As a consequence, 
\emph{the invariant Hilbert scheme is projective over an affine scheme.} 

\medskip

\noindent
(ii) The quotient-scheme map satisfies a natural equivariance property: with the notation and assumptions of 
Proposition \ref{prop:aut}, the $H$-action on $X$ induces an action on $X/\!/G$ so that the quotient morphism
$\pi$ is equivariant. This yields in turn an $H$-action on the punctual Hilbert scheme $\Hilb_n(X/\!/G)$; 
moreover, \emph{the morphism (\ref{eqn:qsm}) is equivariant} 
(as may be checked along the lines of that proposition).

Also, $H$ acts on the symmetric product $(X/\!/G)^{(n)}$ and the morphism (\ref{eqn:hccm}) is $H$-equivariant.

\medskip

\noindent
(iii) In the case that $X$ is a finite-dimensional $G$-module $V$, and $H = (\bG_m)^r$ 
is the center of $\GL(V)^G$, the closed $H$-fixed points of $\Hilb^G_h(X)$ 
may be identified with those $G$-stable ideals $I \subset \cO(V)$ that are homogeneous 
with respect to the isotypical decomposition (\ref{eqn:is}). Moreover, 
\emph{the closure of each $H$-orbit in $\Hilb^G_h(V)$ contains a fixed point}: 
indeed, the closure of each $H$-orbit in $V$ contains the origin, and hence the same holds 
for the induced action of $H$ on the symmetric product $(V/\!/G)^{(n)}$ (where the
origin is the image of $(0, \ldots,0)$). Together with the properness of the morphism
$\delta$ and Borel's fixed point theorem, this implies the assertion.
\end{remarks}

\begin{examples}\label{ex:qsm}
(i) For the toric Hilbert scheme $\Hilb^T(\lambda_1,\ldots,\lambda_N) = \Hilb^T(\ulambda)$ 
of Example \ref{ex:tuf}(i), the quotient-scheme map may be refined as follows. 
Given $\lambda$ in the monoid $\Gamma$ generated by $\lambda_1,\ldots,\lambda_N$, consider the graded subalgebra
$$
\cO(\bA^N)^{(\lambda)} := \bigoplus_{m = 0}^{\infty} \cO(\bA^N)_{m \lambda} \subset \cO(\bA^N)
$$
with degree-$0$ component $\cO(\bA^N)_0 =\cO(\bA^N/\!/T)$. Replacing $\lambda$ with a
positive integer multiple, we may assume that the algebra $\cO(\bA^N)^{(\lambda)}$ is finite 
over its subalgebra generated by its components of degrees $0$ and $1$. Then 
$$
\bA^N/\!/_{\lambda}T := \Proj \big( \cO(\bA^N)^{(\lambda)} \big)
$$
is a projective variety over $\bA^N/\!/_{\lambda}T$, and the twisting sheaf $\cO(1)$
is an ample invertible sheaf on $\bA^N/\!/_{\lambda}T$, generated by its subspace 
$\cO(\bA^N)_{\lambda}$ of global sections. Moreover, one may define a morphism 
$$
\gamma_{\lambda} : \Hilb^T(\ulambda) \longrightarrow \bA^N/\!/_{\lambda}T
$$
which lifts the quotient-scheme morphism $\gamma$. The collection of these morphisms forms 
a finite inverse system; its inverse limit is called the {\bf toric Chow quotient} 
and denoted by $\bA^N/\!/_C T$. This construction yields the {\bf toric Chow morphism}
$$
\Hilb^T(\ulambda) \longrightarrow \bA^N/\!/_C T
$$
which induces an isomorphism on the associated reduced schemes, under some additional
assumptions (see \cite[Section 5]{HS04}). 

\medskip

\noindent
(ii) With the notation of Example \ref{ex:tuf}(ii), we may now show that 
$\Hilb_{h_n}^{\SL_2}\big( V(n) \big)$ is either an affine line if $n = 2$ or $4$, 
or a reduced point for all other values of $n$. 

Indeed, for the natural action of $\bG_m$, each orbit closure contains the unique fixed point $Z$.
If $n \neq 2$ and $n \neq 4$, then it follows that $\Hilb_{h_n}^{\SL_2}\big( V(n) \big)$ is just $Z$,
since its Zariski tangent space at that point is trivial. 

On the other hand, we have constructed a family $\cZ \subset V(2) \times \bA^1$
(Example \ref{ex:fam}) and hence a morphism 
$$
f : \bA^1 \to \Hilb_{h_2}^{\SL_2}\big( V(2) \big).
$$
Moreover, $f$ is injective (on closed points), since the fibers of $\cZ$ are pairwise distinct subschemes 
of $V(2)$. Also, $\cZ$ is stable under the action of $\bG_m$ on $V(2) \times \bA^1$ via 
$t \cdot (x,y) = (tx, t^2y)$ and hence $f$ is $\bG_m$-equivariant for the action on $\bA^1$ via
$t \cdot y = t^2 y$. In particular, $\Hilb_{h_2}^{\SL_2}\big( V(2) \big)$ has dimension $\geq 1$ at $Z$.
Since its Zariski tangent space has dimension $1$, it follows that 
$\Hilb_{h_2}^{\SL_2}\big( V(2) \big)$ is smooth and $1$-dimensional at $Z$. Using the $\bG_m$-action, 
it follows in turn that $f$ is an isomorphism; hence the natural map
$\cZ \to \Univ_{h_2}^{\SL_2}\big( V(2) \big)$ is an isomorphism as well. The quotient-scheme map 
is also an isomorphism in that case.

Likewise, the family $\cW$ of Example \ref{ex:tuf}(ii) yields isomorphisms 
$$
\bA^1 \longrightarrow \Hilb_{h_4}^{\SL_2}\big( V(4) \big),
\quad 
\cW \longrightarrow \Univ_{h_4}^{\SL_2}\big( V(4) \big).
$$
Moreover, the quotient-scheme map is a closed immersion 
$$
\bA^1 \hookrightarrow V(4)/\!/\SL_2 \cong \bA^2.
$$

\medskip

\noindent
(iii) The preceding argument adapts to show that $\Hilb_{h_{\lambda}}^G\big( V(\lambda) \big)$
is either an affine line or a reduced point, with the notation of Example \ref{ex:tuf}(iii).
The quotient-scheme map is again a closed immersion.

\medskip

\noindent
(iv) With the notation of Example \ref{ex:m2r}, we describe the quotient-scheme map $\gamma$;
it takes values in $V/\!/G$ since $h(0) = 1$. 
Observe that the image of $G$ in $\GL\big(V(2)\big) \cong \GL_3$ is the special orthogonal group 
associated with the non-degenerate quadratic form $\Delta$ (the discriminant). By classical invariant 
theory, it follows that the algebra of invariants of $rV(2)$ is generated by the maps
$$
(A_1, \ldots, A_r) \longmapsto \Delta(A_i), \quad \delta(A_i, A_j),  \quad \det(A_i,A_j,A_k),
$$
where $\delta$ denotes the bilinear form associated with $\Delta$. But $\det(A_i,A_j,A_k)$ vanishes
identically on the image of $\gamma$, which may thus be identified with the variety of symmetric
$r \times r$ matrices of rank $\leq 1$. In other words, 
\emph{$\gamma\big( \Hilb^G_h(V) \big)$ is the affine cone over $\bP^{r-1}$ embedded via $\cO_{\bP^{r-1}}(2)$}. 
Moreover, $\gamma$ is a $\GL_r$-equivariant desingularization of that cone, with exceptional locus 
the homogeneous divisor $\Omega_0$. This easily yields an isomorphism of $\GL_r$-varieties
$$
\Hilb^G_h(V) \cong O_{\bP^{r-1}}(-2)
$$
where $O_{\bP^{r-1}}(-2)$ denotes the total space of the line bundle of degree $-2$ over $\bP^{r-1}$,
i.e., the blow-up at the origin of the image of $\gamma$.
\end{examples}

We now apply the construction of the quotient-scheme map to obtain a ``flattening'' of the categorical 
quotient $\pi : X \to X/\!/G$, where \emph{$X$ is an irreducible variety}. Then $X/\!/G$ is an irreducible
variety as well, and there exists a largest open subset $Y_0 \subset Y := X/\!/G$ such that the pull-back 
$$
\pi_0 : X_0 := \pi^{-1}(Y_0) \longrightarrow Y_0
$$
is flat. It follows that the (scheme-theoretic) fibers of $\pi$ at all closed points of its {\bf flat locus} 
$Y_0$ have the same Hilbert function, say $h = h_X$. Since $F/\!/G$ is a (reduced) point for each such 
fiber $F$, we have $h(0) =1$. Thus, the invariant-scheme map associated with this special Hilbert function 
$h$ is just a morphism
$$
\gamma : \Hilb^G_h(X) \longrightarrow X/\!/G.
$$

\begin{proposition}\label{prop:flat}
With the preceding notation and assumptions, the diagram
\begin{equation}\label{eqn:quot}
\CD
\Univ^G_h(X) @>{q}>> X \\
@V{p}VV @V{\pi}VV \\
\Hilb^G_h(X) @>{\gamma}>> X/\!/G \\ 
\endCD
\end{equation}
commutes, where the morphisms from $\Univ^G_h(X)$ are the projections. 
Moreover, the pull-back of $\gamma$ to the flat locus of $\pi$ is an isomorphism.
\end{proposition}

\begin{proof}
Set for simplicity $\cZ:= \Univ^G_h(X)$ and $S := \Hilb^G_h(X)$. 
Then the natural map $\cO_S \to (p_* \cO_{\cZ})^G$ is an isomorphism,
by the definition of the Hilbert scheme and the assumption that $h(0) = 1$.
Thus, $p$ factors as the quotient morphism $\cZ \to \cZ/\!/G$ followed by an isomorphism 
$\cZ/\!/G \to S$. In view of the definition of $\gamma$, this yields the first assertion.

For the second assertion, let $S_0 := \gamma_0^{-1}(Y_0)$ and 
$\cZ_0 := p^{-1}(S_0)$. By the preceding step, we have a closed immersion
$\iota : \cZ_0 \subset X_0 \times_{Y_0} S_0$ of $G$-schemes. But both are flat
over $S_0$ with the same Hilbert function $h$. Thus, the associated sheaves
of covariants $\cF_M$ (for $\cZ_0$) and $\cG_M$ (for $X_0 \times_{Y_0} S_0$)
are locally free sheaves of $\cO_S$-modules of the same rank, and come with a surjective 
morphism $\iota_M^* : \cG_M \to \cF_M$. It follows that $\iota_M^*$ is an isomorphism, 
and in turn that so is $\iota$.
\end{proof}

This construction is of special interest in the case that $G$ is a finite group, see 
\cite[Section 4]{Be08} and also Subsection \ref{subsec:rcqs}. It also deserves further
study in the setting of connected algebraic groups.

\begin{example}
Let $G$ be a semi-simple algebraic group acting on its Lie algebra $\fg$ via the
adjoint representation. By a classical result of Kostant, the categorical quotient
$\fg/\!/G$ is an affine space of dimension equal to the rank $r$ of $G$ (the dimension of
a maximal torus $T \subset G$). Moreover, the quotient morphism $\pi$ is flat; its fibers
are exactly the orbit closures of regular elements of $\fg$ (i.e., those with centralizer
of minimal dimension $r$), and the corresponding Hilbert function $h = h_{\fg}$ is given by
\begin{equation}\label{eqn:hilb}
h(\lambda) = \dim V(\lambda)^T.
\end{equation}
Thus, the invariant-scheme map yields an isomorphism $\Hilb^G_h(\fg) \cong \fg/\!/G$.

If $G = \SL_2$, then $\fg \cong V(2)$ and we get back that 
$\Hilb^{\SL_2}_{h_2}\big( V(2) \big) \cong \bA^1$. More generally, when applied to $G = \SL_n$, 
we recover a result of Jansou and Ressayre, see \cite[Theorem 2.5]{JR09}. They also show that 
$\Hilb^{\SL_n}_h(\mathfrak{sl}_n)$ is an affine space whenever $h$ is the Hilbert function 
of an \emph{arbitrary} orbit closure, and they explicitly describe the universal family 
in these cases; see [loc.~cit., Theorem 3.6].
\end{example}

\section{Some further developments and applications}
\label{sec:sa}

\subsection{Resolution of certain quotient singularities}
\label{subsec:rcqs}

In this subsection, we assume that the group $G$ is \emph{finite}. We discusss a direct construction
of the invariant Hilbert scheme in that setting, as well as some applications; the reader may consult
the notes \cite{Be08} for further background, details, and developments.

Recall that $\Irr(G)$ denotes the set of isomorphism classes of simple $G$-modules. Since this set is finite, 
functions $h : \Irr(G) \to \bN$ may be identified with isomorphism classes of arbitrary finite-dimensional 
modules, by assigning to $h$ the $G$-module $\bigoplus_{M \in \Irr(G)} h(M) M$. 
For example, the function $M \mapsto \dim(M)$ corresponds to the regular representation, i.e.,
$\cO(G)$ where $G$ acts via left multiplication.

Given such a function $h$ and a $G$-scheme $X$, we may consider the invariant Hilbert functor $Hilb^G_h(X)$ 
as in Subsection \ref{subsec:fam}: it associates with any scheme $S$ the set of closed $G$-stable
subschemes $\cZ \subset X \times S$ such that the projection $p: \cZ \to S$ is flat and the module
of covariants $\Hom^G(M,p_* \cO_{\cZ})$ is locally free of rank $h(M)$ for any $M \in \Irr(G)$.

For such a family, the sheaf $p_* \cO_{\cZ}$ is locally free of rank 
$$
n = n(h) := \sum_{M \in \Irr(G)} h(M) \dim(M),
$$
in view of the isotypical decomposition (\ref{eqn:decs}). In other words, $\cZ$ is finite and flat over $S$
of constant degree $m$, the dimension of the representation associated with $h$. 
If $X$ is \emph{quasi-projective}, 
then the punctual Hilbert scheme $\Hilb_n(X)$ exists and is equipped with an action of $G$ 
(see Proposition \ref{prop:aut}). Thus, we have a morphism $f: S \to \Hilb_n(X)$ which is readily seen to be 
$G$-invariant. In other words, $f$ factors through a morphism to the {\bf fixed point subscheme} 
$\Hilb_n(X)^G \subset \Hilb_n(X)$, i.e., the largest closed $G$-stable subscheme on which $G$ acts trivially. 
Moreover, the pull-back of the universal family $\Univ_n(X)$ to $\Hilb_n^G(X)$ is a finite flat family of 
$G$-stable subschemes of $X$, and has a well-defined Hilbert function on each connected component 
(see Remark \ref{rem:fam}(iii)). 

This easily implies the following version of Theorem \ref{thm:hilb} for finite groups:

\begin{proposition}\label{prop:conn}
With the preceding notation and assumptions, the Hilbert functor $Hilb^G_h(X)$ is represented 
by a union of connected components of the fixed point subscheme $\Hilb_n(X)^G$.
\end{proposition}
 
Also, the quotient $\pi : X \to X/G$ exists, where the underlying topological space to $X/G$ is just 
the orbit space, and the structure sheaf $\cO_{X/G}$ equals $(\pi_* \cO_X)^G$. Moreover, the set-theoretical
fibers of $\pi$ are exactly the $G$-orbits. As in Subsection \ref{subsec:qsm}, this yields a 
quotient-scheme map in this setting,
$$
\gamma : \Hilb^G_h(X) \longrightarrow \Hilb_n(X/G).
$$
(In fact, the assignement $\cZ \mapsto \cZ/\!/G$ yields a morphism from $\Hilb_n^G(X)$ to the disjoint union 
of the punctual Hilbert schemes $\Hilb_m(X/G)$ for $m \leq n$).

We now assume that \emph{$X$ is an irreducible variety on which $G$ acts faithfully}. Then $X$ contains
$G$-{\bf regular points}, i.e., points with trivial isotropy groups, and they form an open $G$-stable
subset $X_{\reg}\subset X$. Moreover, the (scheme-theoretic) fiber of $\pi$ at a given $x \in X(k)$ equals
the orbit $G \cdot x$ if and only if $x$ is $G$-regular. In other words, the {\bf regular locus} $X_{\reg}/G$ 
is the largest open subset of $X/G$ over which $\pi$ induces a Galois covering with group $G$; 
it is contained in the flat locus. Thus, the Hilbert function $h_X$ associated with the general fibers 
of $\pi$ (as in Subsection \ref{subsec:qsm}) is just that of the regular representation.
The corresponding invariant Hilbert scheme is called the {\bf $G$-Hilbert scheme} and denoted by 
$G$-$\Hilb(X)$. It is a union of connected components of $\Hilb_n(X)^G$, where $n$ is the order of $G$, 
and comes with a projective morphism
\begin{equation}\label{eqn:hcf}
\gamma : G-\Hilb(X) \longrightarrow X/G
\end{equation}
which induces an isomorphism above the regular locus $X_{\reg}/G$. Moreover, $\gamma$ fits
into a commutative square
$$
\CD
G-\Univ(X) @>{q}>> X \\
@V{p}VV @V{\pi}VV \\
G-\Hilb(X) @>{\gamma}>> X/G \\
\endCD
$$
by Proposition \ref{prop:flat}. In other words, $G$-$\Univ(X)$ is a closed $G$-stable subscheme 
of the fibered product $X \times_{X/G} G-\Hilb(X)$.

We denote by $G$-$\cH_X$ the closure of $\gamma^{-1}(X_{\reg}/G)$ in $G$-$\Hilb(X)$, equipped with the 
reduced subscheme structure.. This is an irreducible component of $G$-$\Hilb(X)$: the 
{\bf main component}, also called the {\bf orbit component}. The points of $G$-$\cH_X$ are the regular 
$G$-orbits and their flat limits as closed subschemes of $X$; also, the quotient-scheme map 
restricts to a projective \emph{birational} morphism $G$-$\cH_X \to X/G$.

\begin{examples}
(i) If $X$ and $X/G$ are smooth, then \emph{the quotient-scheme map (\ref{eqn:hcf}) is an isomorphism}.
(Indeed, $\pi$ is flat in that case, and the assertion follows from Proposition \ref{prop:flat}).

In particular, if $V$ is a finite-dimensional vector space and $G \subset \GL(V)$ a finite
subgroup generated by pseudo-reflections, then $\cO(V)^G$ is a polynomial algebra by a theorem
of Chevalley and Shepherd-Todd. Thus, (\ref{eqn:hcf}) is an isomorphism under that assumption.

\medskip

\noindent
(ii) If $X$ is a smooth surface, then every punctual Hilbert scheme $\Hilb_n(X)$ is smooth.
Since smoothness is preserved by taking fixed points under finite (or, more generally, reductive) 
group actions, it follows that \emph{$\Hilb^G_h(X)$ is smooth for any function $h$.}
Thus, the quotient-scheme map $\gamma : G-\cH_X \to X/G$ is a resolution of singularities.

In particular, if $G$ is a finite subgroup of $\GL_2$ which is not generated by pseudo-reflections, 
then the quotient $\bA^2/G$ is a normal surface with singular locus (the image of) the origin, and
$\gamma : G-\cH_{\bA^2} \to \bA^2/G$ is a canonical desingularization. If in addition $G \subset \SL_2$, 
then $G$ contains no pseudo-reflection, and the resulting singularity is a rational double point.
In that case, $\gamma$ yields the \emph{minimal} desingularization (this result is due to Ito and
Nakamura, see \cite{IN96, IN99}; a self-contained proof is provided in \cite[Section 5]{Be08}).

\medskip

\noindent
(iii) The preceding argument does not extend to smooth varieties $X$ of dimension $\geq 3$,
since $\Hilb_n(X)$ is generally singular in that setting. Yet it was shown by Bridgeland, King and Reid
via homological methods that \emph{$G$-$\Hilb(X)$ is irreducible and has trivial canonical sheaf, if $\dim(X) \leq 3$ 
and the canonical sheaf of $X$ is equivariantly trivial} (see \cite[Theorem 1.2]{BKR01}). As a consequence,
if $G \subset \SL_n$ with $n \leq 3$, then $\gamma : G-\Hilb(\bA^n) \to \bA^n/G$ is a crepant resolution;
in particular, $G -\Hilb(\bA^n)$ is irreducible. This result fails in dimension 4 as the $G$-Hilbert scheme 
may be reducible; this is the case for the binary tetrahedral group $G \subset \SL_2$ acting on $k^4$
via the direct sum of two copies of its natural representation, see \cite{LS08}.

\end{examples}

\subsection{The horospherical family}
\label{subsec:thf}
 
In this subsection, $G$ denotes a \emph{connected} reductive group. We present a classical algebraic
construction that associates with each affine $G$-scheme $Z$, a ``simpler'' affine $G$-scheme $Z_0$ 
called its horospherical degeneration (see \cite{Po86} or \cite[\S 15]{Gr97}). Then we interpret this 
construction in terms of families, after \cite{AB05}.

We freely use the notation and conventions from highest weight theory (Subsection \ref{subsec:rg}) 
and denote by $\alpha_1, \ldots, \alpha_r$ the simple roots of $(G,T)$ associated with the Borel 
subgroup $B$ (i.e., the corresponding positive roots are those of $(B,T)$).

Given a $G$-algebra $A$, recall the isotypical decomposition
\begin{equation}\label{eqn:idec}
A = \bigoplus_{\lambda \in \Lambda^+} A_{(\lambda)} \quad \text{where} \quad
A_{(\lambda)} \cong A^U_{\lambda} \otimes_k V(\lambda).
\end{equation}
Also recall that when $G$ is a torus, $A_{(\lambda)}$ is just the weight space $A_{\lambda}$
and (\ref{eqn:idec}) is a \emph{grading} of $A$, i.e., 
$A_{\lambda} \cdot A_{\mu} \subset A_{\lambda + \mu}$ for all $\lambda$, $\mu$. 
For an arbitrary group $G$, (\ref{eqn:idec}) is no longer a grading, but gives rise to a \emph{filtration} 
of $A$. To see this, we study the multiplicative properties of the isotypical decomposition.

Given $\lambda,\mu \in \Lambda^+$, there is an isomorphism of $G$-modules
\begin{equation}\label{eqn:tens}
V(\lambda) \otimes_k V(\mu) \cong \bigoplus_{\nu \in \Lambda^+} c_{\lambda,\mu}^{\nu} V(\nu)
\end{equation}
where the $c_{\lambda,\mu}^{\nu}$'s are non-negative integers, called the 
{\bf Littlewood-Richardson coefficients}.
Moreover, if $c_{\lambda,\mu}^{\nu} \neq 0$ then $\nu \leq \lambda + \mu$ where $\leq$ is the
partial ordering on $\Lambda$ defined by: 
$$
\mu \leq \lambda \Leftrightarrow \lambda - \mu = \sum_{i=1}^r n_i \alpha_i
\text{ for some non-negative integers } n_1,\ldots, n_r.
$$
Finally, $c_{\lambda,\mu}^{\lambda + \mu} = 1$, i.e., the simple module with the largest weight 
$\lambda + \mu$ occurs in the tensor product $V(\lambda) \otimes_k V(\mu)$ with multiplicity $1$. 
This special component is called the {\bf Cartan component} of $V(\lambda) \otimes_k V(\mu)$.

We set
$$
A_{(\leq \lambda)} := \bigoplus_{\mu \in \Lambda^+, \, \mu \leq \lambda} A_{(\mu)}.
$$
In view of (\ref{eqn:tens}), we have
\begin{equation}\label{eqn:mult}
A_{(\leq \lambda)} \cdot A_{(\leq \mu)} \subset A_{(\leq \lambda + \mu)}
\end{equation}
for all dominant weights $\lambda,\mu$. In other words, the $G$-submodules $A_{(\leq \lambda)}$ form 
an increasing filtration of the $G$-algebra $A$, indexed by the partially ordered group $\Lambda$.
The associated graded algebra $\gr(A)$ is a $G$-algebra, isomorphic to $A$ as a $G$-module 
but where the product of any two isotypical components $A_{(\lambda)}$ and $A_{(\mu)}$ is
obtained from their product in $A$ by projecting on the Cartan component $A_{(\lambda + \mu)}$.
Thus, the product of any two simple submodules of $\gr(A)$ is either their Cartan product, 
or zero. Also, note that $\gr(A)^U \cong A^U$ as $T$-algebras, since 
$A^U_{(\lambda)} = A^U_{\lambda}$ for all $\lambda$.

Now consider the {\bf Rees algebra} associated to this filtration:
\begin{equation}\label{eqn:rees}
\cR(A) := \bigoplus_{\mu \in \Lambda} A_{(\leq \mu)} \, e^{\mu}
= \bigoplus_{\lambda \in \Lambda^+, \, \mu \in \Lambda, \, \lambda \leq \mu} A_{(\lambda)}\, e^{\mu}
\end{equation}
where $e^{\mu}$ denotes the character $\mu$ viewed as a regular function on $T$
(so that $e^{\mu} e^{\nu} = e^{\mu + \nu}$ for all $\mu$ and $\nu$). Thus, $\cR(A)$ is a subspace of 
$$
A \otimes_k \cO(T) = \bigoplus_{\lambda \in \Lambda, \, \mu \in \Lambda} A_{(\lambda)}\, e^{\mu}.
$$
In fact, $A \otimes_k \cO(T)$ is a $G \times T$-algebra, and $\cR(A)$ is a 
$G \times T$-subalgebra by the multiplicative property (\ref{eqn:mult}). Also, note that
$\cR(A)$ contains variables
\begin{equation}\label{eqn:var}
t_1:= e^{\alpha_1}, \ldots, t_r := e^{\alpha_r}
\end{equation}
associated with the simple roots; the monomials in these variables are just
the $e^{\mu -\lambda}$ where $\lambda \leq \mu$. By (\ref{eqn:rees}), we have
$$
\cR(A) \cong \bigoplus_{\lambda \in \Lambda^+} A_{(\lambda)} \, e^{\lambda} \otimes_k k[t_1,\ldots,t_r]
\cong A[t_1,\ldots,t_r]
$$
as $G$-$k[t_1,\ldots,t_r]$-modules. In particular, $\cR(A)$ is a free module over the polynomial ring 
$k[t_1,\ldots,t_r] \subset \cO(T)$. Moreover, we have an isomorphism of $T$-$k[t_1,\ldots,t_r]$-algebras
$$
\cR(A)^U \cong A^U[t_1,\ldots,t_r]
$$
that maps each $f \in \cR(A)^U_{\lambda}$ to $f e^{\lambda}$. Also, the ideal 
$(t_1,\ldots,t_r) \subset \cR(A)$ is $G$-stable, and the quotient by that ideal is just
the $G$-module $A$ where the product of any two components $A_{(\lambda)}$, $A_{(\mu)}$ is 
the same as in $\gr(A)$.  In other words,
$$
\cR(A)/(t_1,\ldots,t_r) \cong \gr(A).
$$ 
On the other hand, when inverting $t_1,\ldots,t_r$, we obtain
$$
\cR(A)[t_1^{-1},\ldots,t_r^{-1}] \cong \bigoplus A_{(\lambda)} \, e^{\mu}
$$
where the sum is over those $\lambda \in \Lambda^+$ and $\mu \in \Lambda$ such that
$\lambda - \mu = n_1 \alpha_1 + \cdots + n_r \alpha_r$ for some integers 
$n_1,\ldots n_r$ (of arbitrary signs). In other words, $\lambda - \mu$ is in
the {\bf root lattice}, i.e., the sublattice of the weight lattice $\Lambda$ generated
by the roots. The torus associated with the root lattice is the {\bf adjoint torus}
$T_{\ad}$, isomorphic to $(\bG_m)^r$ via the homomorphism
$$
\ualpha : T \longrightarrow (\bG_m)^r, \quad 
t \longmapsto \big( \alpha_1(t), \ldots, \alpha_r(t) \big).
$$
Moreover, the kernel of $\ualpha$ is the center of $G$, that we denote by
$Z(G)$. This identifies $T_{\ad}$ with $T/Z(G)$, a maximal torus of the adjoint group 
$G/Z(G)$ (whence the name). Now $Z(G)$ acts on $A \otimes_k \cO(T)$ via its action on $A$ 
as a subgroup of $G$ (then each isotypical component $A_{(\lambda)}$ is an eigenspace 
of weight $\lambda_{\vert Z(G)}$) and its action on $\cO(T)$ as a subgroup of $T$ 
(then each $e^{\mu}$ is an eigenspace of weight $-\mu$). Moreover, the invariant ring satisfies
\begin{equation}\label{eqn:loc}
\big( A \otimes_k \cO(T) \big)^{Z(G)} \cong \cR(A)[t_1^{-1},\ldots,t_r^{-1}]
\end{equation}
as $G \times T$-algebras over $\cO(T)^{Z(G)}\cong k[t_1,t_1^{-1},\ldots,t_r,t_r^{-1}]$.

Translating these algebraic constructions into geometric terms yields the following statement:

\begin{proposition}\label{prop:hor}
Let $Z = \Spec(A)$ be an affine $G$-scheme and $p: \cZ \to \bA^r$ the morphism
associated with the inclusion $k[t_1,\ldots,t_r] \subset \cR(A)$, where $\cR(A)$ denotes
the Rees algebra (\ref{eqn:rees}), and $t_1,\ldots,t_r$ the variables (\ref{eqn:var}).
 
Then $p$ is a flat family of affine $G$-schemes, and the induced family of $T$-schemes 
$\cZ/\!/U \to \bA^r$ is trivial with fiber $Z/\!/U$. The fiber of $p$ at $0$ is 
$\Spec\big( \gr(A) \big)$.

Moreover, the pull-back of $p$ to $T_{\ad}\subset \bA^r$ is isomorphic to the projection
$Z \times^{Z(G)} T \to T/Z(G) = T_{\ad}$. In particular, all fibers of $p$ at general points 
of $\bA^r$ are isomorphic to $Z$.
\end{proposition}

This \emph{special fiber} $Z_0$ is an affine $G$-scheme such that the isotypical decomposition
of $\cO(Z_0)$ is a grading; such schemes are called {\bf horospherical}, and 
$$
Z_0:= \Spec\big( \gr(A) \big)
$$ 
is called the {\bf horospherical degeneration} of the affine $G$-scheme $Z$.
We say that $p: \cZ \to \bA^r$ is the {\bf horospherical family} (this terminology
originates in hyperbolic geometry; note that horospherical varieties need not be
spherical).   

For example, the cone of highest weight vectors 
$Z = \overline{G \cdot v_{\lambda}} = G \cdot v_{\lambda} \cup \{0 \}$ 
is a horospherical $G$-variety, in view of the isomorphism (\ref{eqn:hwc}). 
In that case, the fixed point subscheme $Z^U$ is the highest weight line 
$kv_{\lambda} = V(\lambda)^U$, and thus $Z = G \cdot Z^U$. In fact, the latter property 
characterizes horospherical $G$-schemes:

\begin{proposition}\label{prop:char}
An affine $G$-scheme $Z$ is horospherical if and only if $Z = G \cdot Z^U$
(as schemes).
\end{proposition}

\begin{proof}
First, note that the closed subscheme $Z^U \subset Z$ is stable under the Borel
subgroup $B$; it follows that $G \cdot Z^U$ is closed in $Z$ for an arbitrary 
$G$-scheme $Z$. Indeed, the morphism 
\begin{equation}\label{eqn:mor}
\varphi :  G \times Z^U \longrightarrow Z, \quad (g,z) \longmapsto g \cdot z
\end{equation}
factors as the morphism
$$
\psi : G \times Z^U \longrightarrow G/B \times Z, \quad 
(g,z) \longmapsto (gB, g \cdot z)
$$ 
followed by the projection $G/B \times Z \to Z$. The latter morphism is proper, 
since $G/B$ is complete; moreover, $\psi$ is easily seen to be a closed immersion.

Also, note that the ideal of $G \cdot Z^U$ in $A = \cO(Z)$ is the intersection of
the $G$-translates of the ideal $I$ of $Z^U$. Thus, $Z = G \cdot Z^U$ if and only
if $I$ contains no non-zero simple $G$-submodule of $A$. Moreover, the ideal $I$ is generated 
by the $g \cdot f - f$ where $g \in U$ and $f \in A$.

We now assume that $Z$ is horospherical. Consider a simple $G$-submodule $V(\lambda) \subset A$.
Then $V(\lambda)$ admits a unique lowest weight (with respect to the partial ordering $\leq$), 
equal to $-\lambda^*$, and the corresponding eigenspace is a line. Moreover, the span of the 
$g \cdot v - v$, where $g \in U$ and $v \in V(\lambda)$, is just the sum of all other 
$T$-eigenspaces; we denote that span by $V(\lambda)_+$. 
Since the product $V(\lambda)\cdot V(\mu)$, where $V(\mu)$ is some other simple submodule of $A$, 
is either $0$ or the Cartan product $V(\lambda + \mu)$, we see that  
$$
V(\lambda)_+ \cdot V(\mu) \subset V(\lambda + \mu)_+.
$$
Thus, the sum of the $V(\lambda)_+$ over all simple submodules is an ideal of $A$, 
and hence equals $I$. In particular, $I$ contains no non-zero simple $G$-submodule of $A$.

Conversely, assume that $Z = G \cdot Z^U$, i.e., the morphism (\ref{eqn:mor}) is 
surjective. Note that $\varphi$ is invariant under the action of $U$ via 
$u \cdot (g,z) = (gu^{-1}, z)$, and also equivariant for the action of $G$ on
$G \times Z^U$ via left multiplication on $G$, and for the given action on $Z$. 
Thus, $\varphi$ yields an inclusion of $G$-algebras
$$
\cO(Z) \hookrightarrow \cO(G)^U \otimes_k \cO(Z^U)
$$
where $G$ acts on the right-hand side through its action on $\cO(G)^U$ via left
multiplication. But $\cO(G)^U$ is also a $T$-algebra via right multiplication;
this action commutes with that of $G$, and we have an isomorphism of $G \times T$-modules
$$
\cO(G)^U \cong \bigoplus_{\lambda \in \Lambda^+} V(\lambda)^*
$$
where $G$ acts naturally on each $V(\lambda)^*$, and $T$ acts via its character $\lambda$
(see e.g. \cite[Section 2.1]{Bri10}). 
In particular, the isotypical decomposition of $\cO(G)^U$ is a grading; thus, the
same holds for $\cO(G)^U \otimes_k \cO(Z^U)$ and for $\cO(Z)$.
\end{proof}

Next, we relate the preceding constructions to the invariant Hilbert scheme of a finite-dimensional 
$G$-module $V$. Here it should be noted that the full horospherical family of a closed $G$-subscheme 
$Z \subset V$ need not be a family of closed $G$-subschemes of $V \times \bA^r$ (see e.g. Example
\ref{ex:hor}). Yet \emph{the pull-back of the horospherical family to $T_{\ad}$ may be identified 
with a family of closed $G$-subschemes of $V \times T_{\ad}$} as follows.

By \ref{eqn:loc}, we have an isomorphism of $G \times T$-algebras 
$$ 
\cO(V \times T)^{Z(G)} \cong \cR \big( \cO(V) \big)[t_1^{-1},\ldots,t_r^{-1}]
$$
over $\cO(T_{\ad}) \cong k[t_1^{\pm 1},\ldots,t_r^{\pm 1}]$. Also, each isotypical component
$V_{(\lambda)}= m_{\lambda} V(\lambda)$ yields a subspace 
$$
m_{\lambda} V(\lambda)^* \, e^{\lambda^*} \subset \cO(V \times T)^{Z(G)},
$$
stable under the action of $G \times T$. Moreover, all these subspaces generate 
a $G \times T$-subalgebra of $\cO(V \times T)^{Z(G)}$, equivariantly isomorphic to $\cO(V)$
where $V$ is a $G$-module via the given action, and $T$ acts linearly on $V$ so that each 
$V_{(\lambda)}$ is an eigenspace of weight $-\lambda^*$. Finally, we have an isomorphism
$$
\cO(V \times T)^{Z(G)} \cong \cO(V) \otimes_k \cO(T_{\ad})
$$
over $G \times T$-algebras over $\cO(T_{\ad})$, which translates into an isomorphism
$$
p^{-1}(T_{\ad}) \cong V \times T_{\ad}
$$
of families of $G$-schemes  over $T_{\ad}$. (In geometric terms, we have trivialized
the homogeneous vector bundle $V \times^{Z(G)} T \to T/Z(G)$ by extending the $Z(G)$-action
on $V$ to a $T$-action commuting with that of $G$). 

This construction extends readily to the setting of families, i.e., given a family 
of closed $G$-subschemes $\cZ \subset V \times S$, we obtain a family of closed $G$-subschemes
$\cW \subset V \times T_{\ad} \times S$. By arguing as in the proof of Proposition \ref{prop:aut},
this defines an \emph{action of $T_{\ad}$ on the invariant Hilbert scheme $\Hilb^G_h(V)$}.

In fact, \emph{this action arises from the linear $T$-action on $V$ for which each $V_{(\lambda)}$ 
has weight $-\lambda^*$}: since $\lambda + \lambda^*$ is in the root lattice for any 
$\lambda \in \Lambda$, the induced action of the center $Z(G) \subset T$ coincides with its action 
as a subgroup of $G$, so that $Z(G)$ acts trivially on $\Hilb^G_h(V)$. 

\begin{example}\label{ex:hor}
Let $G = \SL_2$ and $Z = G \cdot xy \subset V(2)$. Then $Z = (\Delta = 1)$ is a closed $G$-subvariety 
of $V(2)$ with Hilbert function $h_2$. One checks that the $G$-submodules $\cO(Z)_{\leq 2n}$ 
are just the restrictions to $Z$ of the spaces of polynomial functions on $V(2)$ with degree $\leq n$. 
Moreover, $Z_0 = \overline{G \cdot y^2}$ and the horospherical family is that of Example
\ref{ex:fam}(ii).

Likewise, if $Z = G \cdot x^2y^2 \subset V(4)$ then $Z_0 = \overline{G \cdot y^4}$ and the 
horospherical family is again that of Example \ref{ex:fam}(ii).

Also, $V(1)$ is its own horospherical degeneration, but the horospherical degeneration of $V(2)$ 
is the singular hypersurface $\{(z,t) \in V(2) \oplus V(0) ~\vert~ \Delta(z) = 0\}$.
\end{example}

\subsection{Moduli of multiplicity-free varieties with prescribed weight monoid}
\label{subsec:mav}

In this subsection, we still consider a connected reductive group $G$, and fix a finitely generated 
submonoid $\Gamma \subset \Lambda^+$. We will construct a moduli space for irreducible multiplicity-free
$G$-varieties $Z$ with weight monoid $\Gamma$ or equivalently, with Hilbert function $h = h_{\Gamma}$ 
(\ref{eqn:hm}). Recall that $Z/\!/U$ is an irreducible multiplicity-free $T$-variety with weight monoid 
$\Gamma$, and hence is isomorphic to $Y := \Spec \big( k[\Gamma] \big)$.

We begin by constructing a special such variety $Z_0$. Choose generators 
$\lambda_1,\ldots,\lambda_N$ of the monoid $\Gamma$. Consider the $G$-module 
$$
V := V(\lambda_1)^* \oplus \cdots \oplus V(\lambda_N)^*,
$$
the sum of highest weight vectors 
$$
v := v_{\lambda_1^*} + \cdots + v_{\lambda_N^*},
$$
and define
$$
Z_0:= \overline{G \cdot v} \subset V.
$$ 
Since $v$ is fixed by $U$, the irreducible variety $Z_0$ is  horospherical in view of Proposition
\ref{prop:char}. The $\Lambda$-graded algebra $\cO(Z_0)$ is generated by $V(\lambda_1), \ldots, V(\lambda_N)$; 
thus, $\Lambda^+(Z_0) = \Gamma$. This yields a special algebra structure on the $G$-module
$$
V(\Gamma) := \bigoplus_{\lambda \in \Gamma} V(\lambda)
$$
such that the subalgebra $V(\Gamma)^U$ is isomorphic to $\cO(Y) = k[\Gamma]$. 

Each irreducible multiplicity-free variety $Z$ with weight monoid $\Gamma$ satisfies $\cO(Z) \cong V(\Gamma)$ 
as $G$-modules and $\cO(Z)^U \cong V(\Gamma)^U$ as $T$-algebras. This motivates the following:

\begin{definition}\label{def:alg} 
A {\bf family of algebra structures of type $\Gamma$ over a scheme $S$} consists of 
the structure of an $\cO_S$-$G$-algebra on $V(\Gamma) \otimes_k \cO_S$ that extends the 
given $\cO_S$-$T$-algebra structure on $V(\Gamma)^U$.  
\end{definition}

In other words, a family of algebra structures of type $\Gamma$ over $S$ is a multiplication law 
$m$ on $V(\Gamma) \otimes_k \cO_S$ which makes it an $\cO_S$-$G$-algebra and restricts to the 
multiplication of the $\cO_S$-$T$-algebra $V(\Gamma)^U \otimes_k \cO_S$. We may write
\begin{equation}\label{eqn:comp}
m = \sum_{\lambda, \mu, \nu \in \Gamma} m_{\lambda, \mu}^{\nu}
\end{equation}
where each component
$$
m_{\lambda, \mu}^{\nu} :  
\big( V(\lambda) \otimes_k \cO_S \big) \otimes_{\cO_S} \big( V(\mu) \otimes_k \cO_S \big)
\longrightarrow V(\nu) \otimes_k \cO_S
$$
is an $\cO_S$-$G$-morphism. Moreover, the commutativity of $m$ and its compatibility with
the multiplication on $V(\Gamma)^U \otimes_k \cO_S$ translate into linear relations between the 
$m_{\lambda,\mu}^{\nu}$'s, while the associativity translates into quadratic relations. Also, each
$m_{\lambda,\mu}^{\nu}$ may be viewed as a linear map
$$
\Hom^G\big( V(\lambda) \otimes_k V(\mu), V(\nu) \big) \longrightarrow H^0(S,\cO_S)
$$
or equivalently, as a morphism of schemes
$$
S \longrightarrow \Hom^G\big( V(\nu), V(\lambda) \otimes_k V(\mu) \big),
$$
and the polynomial relations between the $m_{\lambda,\mu}^{\nu}$'s are equivalent to polynomial
relations between these morphisms. It follows that 
\emph{the functor $M_{\Gamma}$ is represented by an affine scheme $\M_{\Gamma}$}, a closed subscheme of
the infinite-dimensional affine space. Also, note that $m_{\lambda,\mu}^{\nu} = 0$ unless 
$\nu \leq \lambda + \mu$.

In fact, \emph{the scheme $\M_{\Gamma}$ is of finite type}; yet one does not know how to obtain this 
directly from the preceding algebraic description. This finiteness property will rather be derived
from a relation of $\M_{\Gamma}$ to the invariant Hilbert scheme $\Hilb^G_h(V)$ that we now present. 

Given a family of algebra structures of type $\Gamma$ over $S$, the inclusion of $G$-modules 
$V^* \subset V(\Gamma)$ yields a homomorphism of $\cO_S$-$G$-algebras 
$$
\varphi : \Sym(V^*) \otimes_k \cO_S \longrightarrow V(\Gamma) \otimes_k \cO_S.
$$
Moreover, the $\cO_S$-$T$-algebra $V(\Gamma)^U \otimes_k \cO_S$ is generated by the images of the highest 
weight lines $V(\lambda_1)^U, \ldots,V(\lambda_N)^U \subset (V^*)^U$. In particular, the restriction
$$
\Sym\big( (V^*)^U \big) \otimes_k \cO_S \longrightarrow V(\Gamma)^U \otimes_k \cO_S
$$ 
is surjective; thus, $\varphi$ is surjective as well. This defines a family of closed $G$-subschemes 
$\cZ \subset V \times S$ with Hilbert function $h$, such that the sheaf of $\cO_S$-algebras 
$(p_*\cO_{\cZ})^U$ is generated by the preceding highest weight lines. Choosing highest weight vectors 
$v_{\lambda_1},\ldots,v_{\lambda_N}$, we obtain a surjective homomorphism of $\cO_S$-$T$-algebras
$\cO_S[t_1,\ldots,t_N] \to  (p_*\cO_{\cZ})^U$ that maps each $t_i$ to $v_{\lambda_i}$.
Equivalently, we obtain a closed immersion $\cZ/\!/U \hookrightarrow \bA^N \times S$
of families of closed $T$-subschemes with Hilbert function $h$, where $T$ acts linearly
on $\bA^N$ with weights $-\lambda_1,\ldots,-\lambda_N$. This is also equivalent to a morphism
$$
f: S  \to \Hilb^T(\ulambda)
$$ 
where the target is the toric Hilbert scheme of Example \ref{ex:tuf}(i). Now the condition that 
our family of algebra structures extends the given algebra structure on $V(\Gamma)^U$ means
that $f$ maps $S$ to the closed point $Z_0/\!/U$, viewed as a general $T$-orbit closure in $\bA^N$. 

Conversely, given a family of closed $G$-subschemes $\cZ \subset V \times S$ with Hilbert function
$h$ such that $(p_*\cO_{\cZ})^U$ is generated by $v_{\lambda_1},\ldots,v_{\lambda_N}$ and the
resulting morphism $f$ maps $S$ to the point $Z_0/\!/U$, we obtain an isomorphism
of $\cO_S$-$G$-modules $\cO_S \otimes_k V(\Gamma) \cong p_*\cO_{\cZ}$ which restricts to an
isomorphism of $\cO_S$-$T$-algebras $\cO_S \otimes_k V(\Gamma)^U \cong (p_*\cO_{\cZ})^U$. This translates 
into a family of algebra structures of type $\Gamma$ over $S$.

Summarizing, we have the following link between algebra structures and invariant Hilbert schemes:

\begin{theorem}\label{thm:mm}
With the preceding notation and assumptions, there exists an open subscheme 
$$
\Hilb^G_h(V)_0 \subset \Hilb^G_h(V)
$$ 
that parametrizes those families $\cZ$ such that the $\cO_S$-algebra $(p_* \cO_{\cZ})^U$ 
is generated by the image of $(V^*)^U$ under $\varphi$. Moreover, there exists a morphism 
$$
f : \Hilb^G_h(V)_0 \longrightarrow \Hilb^T(\ulambda)
$$ 
that sends $\cZ$ to $\cZ/\!/U$. The fiber of $f$ at the closed point
$Z_0/\!/U \in \Hilb^T(\ulambda)$ represents the functor $M_{\Gamma}$.
\end{theorem}

We denote the fiber of $f$ at $Z_0/\!/U$ by $\M_{\Gamma}$ and call it the 
{\bf moduli scheme of multiplicity-free varieties with weight monoid $\Gamma$}. Since it represents the 
functor $M_{\Gamma}$, the scheme $\M_{\Gamma}$ is independent of the choices of generators of $\Gamma$ and of 
highest weight vectors. It comes with a special point $Z_0$, the common horospherical degeneration to 
all of its closed points. 

We may also consider the preimage in $\Hilb^G_h(V)_0$ of the open $(\bG_m)^N$-orbit 
$\Hilb^T_{\ulambda} \subset \Hilb^T(\ulambda)$ that consists of general $T$-orbit closures. 
This preimage is an open subscheme of $\Hilb^G_h(V)_0$, that we denote by $\Hilb^G_{\ulambda}$. 
Its closed points are exactly the irreducible multiplicity-free varieties $Z \subset V$ 
having weight monoid $\Gamma$ and such that the projections $Z \to V(\lambda_1)^*, \ldots,V(\lambda_N)^*$ 
are all non-zero; equivalently, the horospherical degeneration $Z_0$ is contained in $V$. 
Such varieties $Z$ are called {\bf non-degenerate}. 

Next, we relate these constructions to the action of the adjoint torus $T_{\ad}$ on $\Hilb^G_h(V)$, 
defined in the previous subsection. The torus $(\bG_m)^N$ acts on $\Hilb^G_h(V)$ as the equivariant 
automorphism group of the $G$-module $V$. This action stabilizes the open subschemes $\Hilb^G_h(V)_0$ 
and $\Hilb^G_{\ulambda}$; moreover, $f$ is equivariant for the natural action of $(\bG_m)^N$ on the 
toric Hilbert scheme. Also, note that the $(\bG_m)^N$-orbit $\Hilb^T_{\ulambda}$ is isomorphic 
to $(\bG_m)^N/\ulambda(T)$ where $\ulambda$ denotes the homomorphism (\ref{eqn:ulambda}). 
This yields an action of $T$ on $\M_{\Gamma}$ and one checks that the center $Z(G)$ acts trivially. 
Thus, \emph{$T_{\ad}$ acts on $\M_{\Gamma}$ and each $(\bG_m)^N$-orbit in $\Hilb^G_{\ulambda}$ intersects 
$\M_{\Gamma}$ along a unique $T_{\ad}$-orbit}. 

Given a family $\cZ$ of non-degenerate subvarieties of $V$, one shows that the associated horospherical 
family $\cX$ is a family of non-degenerate subvarieties of $V$ as well. It follows that 
\emph{the $T_{\ad}$-action on $\M_{\Gamma}$ extends to an $\bA^r$-action such that the origin of $\bA^r$
acts via the constant map to $Z_0$.} In particular, $Z_0$ is the unique $T_{\ad}$-fixed point and is 
contained in each $T_{\ad}$-orbit closure; thus, \emph{the scheme $\M_{\Gamma}$ is connected}.
 
The $T_{\ad}$-action on $\M_{\Gamma}$ may also be seen in terms of multiplication laws (\ref{eqn:comp}): 
by \cite[Proposition 2.11]{AB05}, 
\emph{each non-zero component $m_{\lambda, \mu}^{\nu}$ is a $T_{\ad}$-eigenvector of weight $\lambda + \mu - \nu$}
(note that $\lambda + \mu - \nu$ lies in the root lattice, and hence is indeed a character of $T_{\ad}$).

As a consequence, given an irreducible multiplicity-free variety $Z$ with weight monoid $\Gamma$, the
$T_{\ad}$-orbit closure of $Z$ (viewed as a closed point of $\M_{\Gamma}$) has for weight monoid 
the submonoid $\Sigma_Z \subset \Lambda$ generated by the weights $\lambda + \mu - \nu$
where $V(\nu)$ is contained in the product $V(\lambda) \cdot V(\mu) \subset \cO(Z)$. In particular,
\emph{the monoid $\Sigma_Z$ is finitely generated}. Again, it is not known how to obtain this result directly 
from the algebraic definition of the {\bf root monoid} $\Sigma_Z$.

In view of a deep theorem of Knop (see \cite[Theorem 1.3]{Kn96}), the saturation of the monoid $\Sigma_Z$ is free, 
i.e., generated (as a monoid) by linearly independent elements. Equivalently, 
\emph{the normalization of each $T_{\ad}$-orbit closure in $\M_{\Gamma}$ is equivariantly isomorphic
to an affine space on which $T_{\ad}$ acts linearly}. 

We also mention a simple relation between the Zariski tangent space to $\M_{\Gamma}$ at a closed point $Z$ 
and the space $T^1(Z)$ parametrizing first-order deformations of $Z$: namely,
\emph{the normal space to the orbit $T_{\ad} \cdot Z$ in $\M_{\Gamma}$ is isomorphic to the $G$-invariant
subspace $T^1(Z)^G$} (see \cite[Proposition 1.13]{AB05}). In particular, 
$$
T_{Z_0} \M_{\Gamma} = T^1(Z_0)^G
$$
as $Z_0$ is fixed by $T_{\ad}$.

In fact, many results of this subsection hold in the more general setting where the algebra $k[\Gamma]$ 
is replaced with an arbitrary $T$-algebra of finite type; see \cite{AB05}. The multiplicity-free case
presents remarkable special features; namely, finiteness properties that will be surveyed in the next subsection.

\begin{example}\label{ex:free}
If the monoid $\Gamma$ is free, then of course we choose $\lambda_1,\ldots,\lambda_N$ to be its minimal generators.
Since they are linearly independent, $\Hilb^T(\ulambda)$ is a (reduced) point and hence
$$
\Hilb^G_h(V)_0 = \Hilb^G_{\ulambda} = \M_{\Gamma}.
$$
Also, since the homomorphism $\ulambda$ is surjective, each $(\bG_m)^N$-orbit in $\Hilb^G_{\ulambda}$ 
is a unique $T_{\ad}$-orbit. The pull-back $\pi : \Univ_{\Gamma} \to \M_{\Gamma}$ of the universal family of 
$\Hilb^G_h(V)$ may be viewed as the universal family of non-degenerate spherical subvarieties of $V$.
\end{example}

\subsection{Finiteness properties of spherical varieties}
\label{subsec:fpsv}

In this subsection, we survey finiteness and uniqueness results relative to the structure and classification 
of spherical varieties. We still denote by $G$ a connected reductive group; we fix a Borel subgroup 
$B \subset G$ and a maximal torus $T \subset B$.   

Recall that a (possibly non-affine) $G$-variety $X$ is spherical, if $X$ is normal and contains an
open $B$-orbit; in particular, $X$ contains an open $G$-orbit $X_0$. Choosing a base point $x \in X_0$ and 
denoting by $H$ its isotropy group, we may identify $X_0$ with the homogeneous space $G/H$. 
We say that $H$ is a {\bf spherical subgroup} of $G$, and the pair $(X,x)$ is an 
{\bf equivariant embedding} of $G/H$; the complement $X \setminus X_0$ is called the {\bf boundary}. 
Morphisms of embeddings are defined as those equivariant morphisms that preserve base points. 
If the variety $X$ is complete, then $X$ is called an {\bf equivariant completion} (or equivariant 
compactification) of $G/H$.

One can show that \emph{any spherical $G$-variety contains only finitely many $B$-orbits}; as a consequence,
\emph{any equivariant embedding of a spherical $G$-homogeneous space contains only finitely many $G$-orbits}. 
Conversely, if a $G$-homogeneous space $X_0$ satisfies the property that all of its equivariant embeddings 
contain finitely many orbits, then $X_0$ is spherical.

Spherical homogeneous spaces admit a further remarkable characterization, in terms of the existence of 
equivariant completions with nice geometric properties. Specifically, consider an embedding $X$ of a homogeneous 
space $X_0 = G/H$. Assume that $X$ is smooth and $X \setminus X_0$ is a union of smooth prime
divisors that intersect transversally; in other words, the boundary is a {\bf divisor with simple normal crossings}. 
Then we may consider the associated {\bf sheaf of logarithmic vector fields}, consisting of those derivations of 
$\cO_X$ that preserve the ideal sheaf of $D := X \setminus X_0$. This subsheaf, denoted by $T_X(-\log D)$, 
is a locally free subsheaf of the tangent sheaf $T_X$ of all derivations of $\cO_X$; both sheaves coincide
along $X_0$. The {\bf logarithmic tangent bundle} is the vector bundle on $X$ associated with $T_X(-\log D)$.
The $G$-action on $(X,D)$ yields an action of the Lie algebra $\fg$ by derivations that preserve $D$,
i.e., a homomorphism of Lie algebras 
$$
\fg \longrightarrow H^0 \big(X, T_X(-\log D) \big).
$$ 
We say that the pair $(X,D)$ is {\bf log homogeneous under $G$} if $\fg$ generates the sheaf 
of logarithmic vector fields. Now \emph{any complete log homogeneous $G$-variety is spherical; moreover, 
any spherical $G$-homogeneous space admits a log homogeneous equivariant completion} (as follows from 
\cite[Sections 2.2, 2.5]{BB96}; see \cite{Bri07b} for further developments on log homogeneous varieties and 
their relation to spherical varieties). We will need a stronger version of part of this result:

\begin{lemma}\label{lem:lh}
Let $X$ be a smooth spherical $G$-variety with boundary a divisor $D$ with simple normal crossings and 
denote by $S_X$ the subsheaf of $T_X(- \log D)$ generated by $\fg$. If $S_X$ is locally free, then 
$S_X = T_X(- \log D)$. In particular, $X$ is log homogeneous.
\end{lemma}

\begin{proof}
Clearly, $S_X$ and $T_X(- \log D)$ coincide along the open $G$-orbit. Since these sheaves are locally free, 
the support of the quotient $T_X(- \log D)/S_X$ has pure codimension $1$ in $X$. But this support is $G$-stable, 
and contains no $G$-orbit of codimension $1$ by \cite[Section 2]{BB96}. Thus, this support is empty; this yields 
our assertion.
\end{proof}

Log homogeneous pairs satisfy an important rigidity property, namely,
\begin{equation}\label{eqn:rig}
H^1\big( X, T_X(-\log D) \big) = 0
\end{equation}
whenever $X$ is complete (as follows from a vanishing theorem due to Knop, see \cite[Theorem 4.1]{Kn94}).
This is the main ingredient for proving the following finiteness result (\cite[Theorem 3.1]{AB05}):

\begin{theorem}\label{thm:fin}
For any $G$-variety $X$, only finitely many conjugacy classes of spherical subgroups of $G$ occur as isotropy 
groups of points of $X$.
\end{theorem}

In other words, \emph{any $G$-variety contains only finitely many isomorphism classes of spherical $G$-orbits}. 

For the proof, one reduces by general arguments of algebraic transformation groups to the case that 
$X$ is irreducible and admits a geometric quotient 
$$
p: X \longrightarrow S
$$ 
where the fibers of $p$ are exactly the $G$-orbits. Arguing by induction on the dimension, 
we may replace $S$ with an open subset; thus, we may further 
assume that $X$ and the morphism $p$ are smooth. Then the spherical $G$-orbits form an open subset of $X$, 
since the same holds for the $B$-orbits of dimension $\dim(X) - \dim(S)$. So we may assume that all fibers are 
spherical. Now, by general arguments of algebraic transformation groups again, there exists an equivariant
fiberwise completion of $X$, i.e., a $G$-variety $\bar{X}$ equipped with a proper $G$-invariant morphism 
$$
\overline{p}: \bar{X} \longrightarrow S
$$ 
such that $\bar{X}$ contains $X$ as a $G$-stable open subset, and $\bar{p}$
extends $p$. We may further perform equivariant blow-ups and hence assume that $\bar{X}$ is smooth, 
the boundary $\bar{X} \setminus X$ is a divisor with simple normal crossings, and the subsheaf 
$S_{\bar{X}} \subset T_{\bar{X}}$ generated by $\fg$ is locally free. By Lemma \ref{lem:lh}, 
it follows that $\bar{X}$ is a family of log homogeneous varieties (possibly after shrinking $S$ again). 
Now the desired statement is a consequence of rigidity (\ref{eqn:rig}) together with arguments of 
deformation theory; see \cite[pp. 113--115]{AB05} for details.

As a direct consequence of Theorem \ref{thm:fin}, \emph{any finite-dimensional $G$-module $V$ contains
only finitely many closures of spherical $G$-orbits, up to the action of $\GL(V)^G$} (see \cite[p. 116]{AB05}).
In view of the results of Subsection \ref{subsec:mav}, it follows that
\emph{every moduli scheme $\M_{\Gamma}$ contains only finitely many $T_{\ad}$-orbits}. In particular, 
\emph{up to equivariant isomorphism, there are only finitely many affine spherical varieties having a 
prescribed weight monoid}.

This suggests that spherical varieties may be classified by combinatorial invariants. Before presenting 
a number of results in this direction, we associate three such invariants to a spherical homogeneous space 
$X_0 = G/H$.   

The first invariant is the set of weights of $B$-eigenvectors in the field of rational functions $k(X_0) = k(G)^H$; 
this is a subgroup of $\Lambda$, denoted by $\Lambda(X_0)$ and called the {\bf weight lattice} of $X_0$. 
The rank of this lattice is called the {\bf rank} of $X_0$ and denoted by $\rk(X_0)$. 
Note that any $B$-eigenfunction is determined by its weight up to a non-zero scalar, since $k(X_0)^B = k$ 
as $X_0$ contains an open $B$-orbit. 

The second invariant is the set $\cV(X_0)$ of those discrete valuations of the field $k(X_0)$, with values in 
the field $\bQ$ of rational numbers, 
that are invariant under the natural $G$-action. One can show that any such valuation is uniquely determined by 
its restriction to $B$-eigenvectors; moreover, this identifies $\cV(X_0)$ to a convex polyhedral cone in 
the rational vector space $\Hom\big( \Lambda(X_0), \bQ \big)$. Thus, $\cV(X_0)$ is called the {\bf valuation cone}.

The third invariant is the set $\cD(X_0)$ of $B$-stable prime divisors in $X_0$, called {\bf colors}; 
these are exactly the irreducible components of the complement of the open $B$-orbit. Any $D \in \cD(X_0)$ 
defines a discrete valuation of $k(X_0)$, and hence (by restriction to $B$-eigenvectors) a point 
$\varphi_D \in \Hom\big( \Lambda(X_0), \bQ \big)$. Moreover, the stabilizer of $D$ in $G$ is a parabolic subgroup 
$G_D$ containing $B$, and hence corresponds to a set of simple roots. Thus, $\cD(X_0)$ may be viewed as 
an abstract finite set equipped with maps $D \mapsto \varphi_D$ to $\Hom\big( \Lambda(X_0), \bQ \big)$
and $D \mapsto G_D$ to subsets of simple roots. 

The invariants $\Lambda(X_0)$, $\cV(X_0)$, $\cD(X_0)$ are the main ingredients of a 
\emph{classification of all equivariant embeddings of $X_0$}, due to Luna and Vust 
(see \cite{Pe10} for an overview, and \cite{Kn91} for an exposition; the original article \cite{LV83} addresses
embeddings of arbitrary homogeneous spaces, see also \cite{Ti07}). This classification is formulated in terms of 
{\bf colored fans}, combinatorial objects that generalize the fans of toric geometry. Indeed, toric varieties are 
exactly the equivariant embeddings of a torus $T$ viewed as a homogeneous space under itself. In that case, 
$\Lambda(T)$ is just the character lattice, and the set $\cD(T)$ is empty; one shows that $\cV(T)$ is the whole space 
$\Hom \big( \Lambda(T),\bQ \big)$.

Another important result, due to Losev (see \cite[Theorem 1]{Lo09a}), asserts that 
\emph{any spherical homogeneous space is uniquely determined by its weight lattice, valuation cone and colors,
up to equivariant isomorphism}. 
The proof combines many methods, partial classifications, and earlier results, including the Luna-Vust
classification and the finiteness theorem \ref{thm:fin}.

Returning to an affine spherical variety $Z$, one can show that the valuation cone of the open $G$-orbit $Z_0$ 
is dual (in the sense of convex geometry) to the cone generated by the root monoid $\Sigma_Z$.
Also, recall that the saturation of $\Sigma_Z$ is a free monoid; its generators are called the {\bf spherical roots}
of $Z$. By another uniqueness result of Losev (see \cite[Theorem 1.2]{Lo09b}), 
\emph{any affine spherical $G$-variety is uniquely determined by its weight monoid and spherical roots, 
up to equivariant isomorphism}. Moreover, 
\emph{any smooth affine spherical $G$-variety is uniquely determined by its weight monoid},
again up to equivariant isomorphism (\cite[Theorem 1.3]{Lo09b}). 

Note that all smooth affine spherical varieties are classified in \cite{KS06}; yet one does not know how to deduce 
the preceding uniqueness result (a former conjecture of Knop) from that classification. 

\begin{example}\label{ex:sph}
The spherical subgroups of $G = \SL_2$ are exactly the closed subgroups of positive dimension. 
Here is the list of these subgroups up to conjugation in $G$:

\medskip

\noindent
(i) $H = B$ (the Borel subgroup of upper triangular matrices of determinant $1$).
Then $G/H \cong \bP^1$ has rank $0$ and a unique color, the $B$-fixed point $\infty$. 

\medskip

\noindent
(ii) $H = U \mu_n$ where $U$ denotes the unipotent part of $B$, and $\mu_n$ the group of diagonal
matrices with eigenvalues $\zeta,\zeta^{-1}$ where $\zeta^n = 1$; here $n$ is a positive integer. 
Then $H \subset B$ and via the resulting map $G/H \to G/B$, we see that $G/H$ is the total space 
of the line bundle $O_{\bP^1}(n)$ minus the zero section. Moreover, $G/H$ has rank $1$ and a unique color, 
the fiber at $\infty$ of the projection to $\bP^1$. We have 
$\Lambda(G/H) = n \bZ \subset \bZ = \Lambda$, and the valuation cone is the whole space 
$\Hom\big( \Lambda(G/H), \bQ \big) \cong \bQ$. 

A smooth equivariant completion of $G/H$ is $\bP \big( O_{\bP^1} \oplus O_{\bP^1}(n) \big)$,
the rational ruled surface of index $n$. By contracting the unique curve of negative 
self-intersection, i.e., the section of self-intersection $-n$, we obtain another equivariant
completion which is singular if $n \neq 1$, and isomorphic to $\bP\big( V(1) \oplus V(0) \big) \cong \bP^2$
if $n = 1$.

\medskip

\noindent
(iii) $H = T$ (the torus of diagonal matrices of determinant $1$). Then 
$G/H \cong G \cdot xy \subset V(2)$ is the affine quadric $(\Delta = 1)$; it has rank $1$ 
and weight lattice $2 \bZ$. Note that $H$ is the intersection of $B$ with the opposite Borel
subgroup $B^-$ (of lower triangular matrices). Thus, $G/H$ admits $\bP^1 \times \bP^1$ as an 
equivariant completion via the natural morphism $G/H \to G/B \times G/B^-$; this is in fact the unique
non-trivial equivariant embedding. Also, $G/H$ has two colors $D^+, D^-$ (the fibers at $\infty$ of the 
two morphisms to $\bP^1$). The valuation cone is the negative half-line in 
$\Hom\big( \Lambda(G/H), \bQ \big) \cong \bQ$, and $D^+,D^-$ are mapped to the same point of the positive 
half-line.

\medskip

\noindent
(iv) $H = N_G(T)$ (the normalizer of $T$ in $G$). Then $G/H \cong G \cdot [xy] \subset \bP\big( V(2) \big)$ 
is the open affine subset $(\Delta = 0)$ in the projective plane $\bP\big( V(2) \big)$, which is the unique
non-trivial embedding. Moreover, $G/H$ has rank $1$ and weight lattice $4 \bZ$. There is a unique color $D$, 
with closure the projective line $\bP\big( y V(1) \big) \subset \bP\big( V(2) \big)$. The valuation cone is 
again the negative half-line in $\Hom\big( \Lambda(G/H), \bQ \big) \cong \bQ$, and $D$ is mapped to a point 
of the positive half-line.

\end{example}

\subsection{Towards a classification of wonderful varieties}
\label{subsec:cwv}

In this subsection, we introduce the class of wonderful varieties, which play an essential role 
in the structure of spherical varieties. Then we present a number of recent works that classify 
wonderful varieties (possibly with additional assumptions) via Lie-theoretical or geometric methods.  

\begin{definition}

A variety $X$ is called {\bf wonderful} if it satisfies the following properties:

\begin{enumerate}

\item{}
$X$ is smooth and projective.

\item{}
$X$ is equipped with an action of a connected reductive group $G$ having an open orbit $X_0$.

\item{}
The boundary $D := X \setminus X_0$ is a divisor with simple normal crossings, and its irreducible
components $D_1,\ldots,D_r$ meet.

\item{}
The $G$-orbit closures are exactly the partial intersections of $D_1,\ldots,D_r$.

\end{enumerate}

\end{definition}

Then $D_1,\ldots,D_r$ are called the \emph{boundary components}; their number $r$ is the \emph{rank} 
of $X$. By (4), $X$ has a unique closed orbit 
$$
Y := D_1\cap \cdots \cap D_r.
$$

The wonderful $G$-varieties of rank $0$ are just the complete $G$-homogeneous varieties, i.e., the
homogeneous spaces $G/P$ where $P \subset G$ is a parabolic subgroup containing $B$. Those of rank $1$ 
are exactly the smooth equivariant completions of a homogeneous space by a homogeneous divisor; they have been 
classified by Akhiezer (see \cite{Ak83}) and they turn out to be spherical. The latter property extends to all ranks: 
in fact, \emph{the wonderful varieties are exactly the complete log homogeneous varieties having a
unique closed orbit}, as follows from \cite{Lu96} combined with \cite[Propositions 2.2.1 and 2.5]{BB96}. 
Moreover, the rank of a wonderful variety coincides with the rank of its open $G$-orbit.

The combinatorial invariants $\Lambda(X_0)$, $\cV(X_0)$, $\cD(X_0)$ associated with the open $G$-orbit
of a wonderful variety $X$ admit simple geometric interpretations. To state them, let $X_1$ denote the 
complement in $X$ of the union of the closures of the colors. Then $X_1$ is an open 
$B$-stable subset of $X$. One shows that $X_1$ is isomorphic to an affine space and meets each $G$-orbit 
in $X$ along its open $B$-orbit; it easily follows that 
\emph{the closure in $X$ of each color $D \in \cD(X_0)$ is a base-point-free divisor, and these divisors form 
a basis of the Picard group of $X$}. In particular, we have equalities in $\Pic(X)$:
\begin{equation}\label{eqn:pic}
D_i = \sum_{D \in \cD} c_{D,i} \, D \quad (i=1,\ldots,r)
\end{equation}
where the $c_{D,i}$ are uniquely determined integers. Also, $X_1 \cap Y$ is the open Bruhat cell in $Y$, 
and hence equals $B \cdot y$ for a unique $T$-fixed point $y \in Y$. Thus, $T$ acts in the normal space 
$T_y (X)/T_y (Y)$ of dimension $r = \rk(X)$; one shows that the corresponding weights $\sigma_1,\ldots,\sigma_r$ 
(called the {\bf spherical roots} of the wonderful variety $X$) are linearly independent. Now 
\emph{the spherical roots form a basis of $\Lambda(X_0)$, and generate the dual cone to $\cV(X_0)$}. 
The dual basis of $\Hom \big( \Lambda(X_0), \bQ \big)$ consists of the opposites of the valuations $v_1,\ldots,v_r$
associated with the boundary divisors. Moreover, (\ref{eqn:pic}) implies that the map 
$\varphi : \cD \to \Hom \big( \Lambda(X_0), \bQ \big)$ is given by 
$$
\varphi(v_D) = \sum_{i=1}^r c_{D,i} \, v_i \quad (D \in \cD).
$$

To each spherical homogeneous space $X_0 = G/H$, one associates a wonderful variety as follows.
Denote by $N_G(H)$ the normaliser of $H$ in $G$, so that the quotient group $N_G(H)/H$ is isomorphic 
to the equivariant automorphism group $\Aut^G(X_0)$. Since $X_0$ is spherical, the algebraic group
$N_G(H)/H$ is diagonalizable; moreover, $N_G(H)$ equals the normalizer of the Lie algebra $\fh$. 
Thus, the homogeneous space $G/N_G(H)$ is the $G$-orbit of $\fh$ viewed as a point of the Grassmannian 
variety $\Gr(\fg)$ of subspaces of $\fg$ (or alternatively, of the scheme of Lie subalgebras of $\fg$). 
The orbit closure 
$$
X:= \overline{G \cdot \fh} \subset \Gr(\fg)
$$ 
is a projective equivariant completion of $G/N_G(H)$, called the \emph{Demazure embedding} of that
homogeneous space. In fact, \emph{the variety $X$ is wonderful} by a result of Losev (see \cite{Lo09c})
based on earlier results of several mathematicians, including Demazure and Knop (see \cite[Corollary 7.2]{Kn96}). 
Moreover, by embedding theory of spherical homogeneous spaces, \emph{the log homogeneous embeddings 
of $G/H$ are exactly those smooth equivariant embeddings that admit a morphism to $X$; then the 
logarithmic tangent bundle is the pull-back of the tautological quotient bundle on $\Gr(\fg)$}.
Also, by embedding theory again, \emph{a complete log homogeneous variety $X'$ is wonderful if and only 
if the morphism $X' \to X$ is finite.}

It follows that every spherical homogeneous space $G/H$ such that $H = N_G(H)$ admits a wonderful
equivariant completion; in the converse direction, if $G/H$ admits such a completion $X$, 
then $X$ is unique, and the quotient $N_G(H)/H$ is finite. In particular, the center of $G$ acts on $X$ 
via a finite quotient; thus, one may assume that $G$ is semi-simple when considering wonderful $G$-varieties.

Since the $G$-variety $\Gr(\fg)$ contains only finitely many isomorphism classes of spherical $G$-orbits, 
and any $G$-homogeneous space admits only finitely many finite equivariant coverings, we see that
\emph{the number of isomorphism classes of wonderful $G$-varieties is finite} (for a given group $G$). 
Also, note that the wonderful varieties are exactly those log homogeneous varieties that are 
{\bf log Fano}, i.e., the determinant of the logarithmic tangent sheaf is ample.

To classify wonderful $G$-varieties, it suffices to characterize those triples $(\Lambda,\cV,\cD)$
that occur as combinatorial invariants of their open $G$-orbits, in view of Losev's uniqueness result. 
In fact, part of the information contained in such triples is more conveniently encoded by abstract 
combinatorial objects called {\bf spherical systems}. These were introduced by Luna, who obtained a complete 
classification of wonderful $G$-varieties for those groups $G$ of type $A$, in terms of spherical systems only.
For an arbitrary group $G$, Luna also showed how to reduce the classification of spherical $G$-homogeneous spaces
to that of wonderful $G$-varieties, and he conjectured that 
\emph{wonderful varieties are classified by spherical systems} (see \cite{Lu01}).

Luna's Lie theoretic methods were further developed by Bravi and Pezzini to classify wonderful varieties in classical 
types $B,C,D$ (see \cite{BP05, BP09}); the case of exceptional types $E$ was treated by Bravi in \cite{Bra07}. 
Thus, Luna's conjecture has been checked in almost all cases. The exceptional type $F_4$ is considered by Bravi and 
Luna in \cite{BL08}; they listed the 266 spherical systems in that case, and they constructed many examples of 
associated wonderful varieties.

Luna's conjecture has also been confirmed for those wonderful $G$-varieties that arise as orbit closures
in projectivizations of simple $G$-modules, via a classification due to Bravi and Cupit-Foutou (see \cite{BC10}). 
These wonderful varieties are called {\bf strict}; they are characterized by the property that the isotropy 
group of each point equals its normalizer, as shown by Pezzini (see \cite[Theorem 2]{Pe07}).
In \cite{BC08}, Bravi and Cupit-Foutou applied that classification to explicitly describe certain
moduli schemes of spherical varieties with a prescribed weight monoid; we now survey their results.   

We say that a submonoid $\Gamma \subset \Lambda^+$ is $G$-{\bf saturated}, if 
$$
\Gamma = \Lambda^+ \cap \bZ \Gamma
$$ 
where $\bZ \Gamma$ denotes the subgroup of $\Lambda$ generated by $\Gamma$. Then $\Gamma$ is finitely generated,
and saturated in the sense arising from toric geometry. By \cite{Pa97}, the $G$-saturated submonoids of
$\Lambda^+$ are exactly the weight monoids of those affine horospherical $G$-varieties $Z_0$ such that 
$Z_0$ is normal and contains an open $G$-orbit with boundary of codimension $\geq 2$. 

Next, fix a $G$-saturated submonoid $\Gamma \subset \Lambda^+$ that is freely generated, with basis
$\lambda_1,\ldots,\lambda_N$. Consider the associated moduli scheme $\M_{\Gamma}$ 
equipped with the action of the adjoint torus $T_{\ad}$ (Subsection \ref{subsec:mav}).
Then $\M_{\Gamma}$ is an open subscheme of $\Hilb^G_h(V)$ by Example \ref{ex:free}, where 
$V = V(\lambda_1)^* \oplus \cdots \oplus V(\lambda_N)^*$ and $h = h_{\Gamma}$.

By the results of \cite[Section 2.3]{BC08}, \emph{$\M_{\Gamma}$ is isomorphic to a $T_{\ad}$-module with linearly 
independent weights, say $\sigma_1,\ldots,\sigma_r$. Moreover, the union of all non-degenerate $G$-subvarieties 
$$
Z \subset V(\lambda_1)^* \oplus \cdots \oplus V(\lambda_N)^*
$$ 
with weight monoid $\Gamma$ is the affine multi-cone over a wonderful $G$-variety 
$$
X \subset  \bP \big( V(\lambda_1)^* \big) \times \cdots \times \bP\big( V(\lambda_N)^* \big)
$$
which is strict and has spherical roots $\sigma_1,\ldots,\sigma_r$.} The closures of the colors 
of $X$ are exactly the pull-backs of the $B$-stable hyperplanes in $\bP \big( V(\lambda_i)^* \big)$
(defined by the highest weight vectors of $V(\lambda_i)$) under the projections
$X \to \bP \big( V(\lambda_i)^* \big)$. 

To prove these results, one first studies the tangent space to $\M_{\Gamma}$ at the horospherical degeneration
$Z_0$, based on \ref{prop:orb}. This $T_{\ad}$-module turns out to be multiplicity-free with weights 
$\sigma_1,\ldots,\sigma_r$ among an explicit list of spherical roots. Then one shows that the data of 
$\lambda_1,\ldots,\lambda_N$ and $\sigma_1,\ldots,\sigma_r$ define a spherical system; finally, by the 
classification of strict wonderful varieties, this spherical system corresponds to a unique such variety $X$.  

Yet several $G$-saturated monoids may well yield the same strict wonderful variety, for instance in the case
that $\Gamma$ has basis a dominant weight $\lambda$ (any such monoid is $G$-saturated); see the final example below.
 
Another natural example of a $G$-saturated monoid is the whole monoid $\Lambda^+$ of dominant weights.
The affine spherical varieties $Z$ having that weight monoid are called a {\bf model $G$-variety}, as every
simple $G$-module occurs exactly once in $\cO(Z)$; then the horospherical degeneration of $Z$ is $Z_0 = G/\!/U$.
The strict wonderful varieties associated with model varieties have been described in detail by Luna
(see \cite{Lu07}).

More recently, Cupit-Foutou generalizes the approach of \cite{BC08} in view of a geometric classification of wonderful 
varieties and of a proof of Luna's conjecture in full generality (see \cite{Cu09}). For this, she associates with any 
wonderful variety of rank $r$ a family of (affine) spherical varieties over the affine space $\bA^r$, having a 
prescribed free monoid. Then she shows that this family is the universal family.

The first step is based on the construction of the {\bf total coordinate ring} (also called the {\bf Cox ring})
of a wonderful variety $X$. Recall that the set $\cD$ of (closure of) colors freely generates the Picard group of 
$X$, and consider the $\bZ^{\cD}$-graded ring
$$
R(X) := \bigoplus_{(n_D) \in \bZ^{\cD}} H^0 \big( X, \cO_X( \sum_{D \in \cD} n_D D) \big)
$$
relative to the multiplication of sections. We may assume that $G$ is semi-simple and simply connected; then 
each space $H^0 \big( X, \cO_X(\sum_{D \in \cD} n_D D)$ has a unique structure of a $G$-module, and the total 
coordinate ring $R(X)$ is a $\bZ^{\cD}$-graded $G$-algebra. It is also a finitely generated unique factorization 
domain. Moreover, the invariant subring $R(X)^U$ is freely generated by the canonical sections $s_D$ of the colors
and by those $s_1, \ldots, s_r$ of the boundary divisors; each $s_i$ is homogeneous of weight $(c_{D,i})_{D \in \cD}$.
Each $s_D$ is a $B$-eigenvector of weight (say) $\omega_D$, and hence generates a simple submodule
$$ 
V_D \cong V(\omega_D) \subset H^0 \big( X, \cO_X(D) \big).
$$
As a consequence, the graded ring $R(X)$ is generated by $s_1,\ldots,s_r$ and by the $V_D$ where $D \in \cD$.
Moreover, $s_1,\ldots,s_r$ form a regular sequence in $R(X)$ (see \cite[Section 3.1]{Bri07a}).

In geometric terms, the affine variety $\tX := \Spec\big( R(X) \big)$ is equipped with an action of the 
connected reductive group $\tG := G \times (\bG_m)^{\cD}$ and with a flat, $G$-invariant morphism
\begin{equation}\label{eqn:tcf}
\pi = (s_1,\ldots,s_r) : \tX \longrightarrow \bA^r
\end{equation}
which is also $(\bG_m)^{\cD}$-equivariant for the linear action of that torus on $\bA^r$ with weights
$\sum_{D\in \cD} c_{D,i} \varepsilon_D$ where $i = 1,\ldots, r$ and $\varepsilon_D :(\bG_m)^{\cD} \to \bG_m$ 
denotes the coefficient on $D$. Moreover, the $\tG$-subvariety $\tX$ is spherical and equipped with a closed 
immersion into the $\tG$-module $\big( \bA^r \times \prod_{D \in \cD}  V_D\big)^*$
that identifies $\pi$ with the projection to $(\bA^r)^* \cong \bA^r$. Here $(\bG_m)^{\cD}$ acts on 
$\prod_{D \in \cD} V_D^*$ via multiplication by $-\varepsilon_D$ on $V_D^*$. 

It follows that $\pi$ may be viewed as a family of non-degenerate spherical $G \times C$-subvarieties of 
$$
V := \bigoplus_{D \in \cD}  V_D^*
$$
where $C$ denotes the neutral component of the kernel of the homomorphism 
$$
(\bG_m)^{\cD} \longrightarrow (\bG_m)^r, \quad 
(t_D)_{D \in \cD} \longmapsto \big( \prod_{D \in \cD} t_D^{c_{D,i}} \big)_{i=1,\ldots,r}.
$$ 
Thus, $C$ is a torus, and $G \times C$ a connected reductive group with maximal torus $T \times C$ and adjoint 
torus $T_{\ad}$. The weight monoid $\Gamma$ is freely generated by the restrictions to $T \times C$
of the weights $(\omega_D,\varepsilon_D)$ of $T \times (\bG_m)^{\cD}$, where $D \in \cD$.

Now the main results of \cite{Cu09} assert that the moduli scheme $\M_{\Gamma}$ is isomorphic to $\bA^r$, and 
$\tX$ to the universal family. Moreover, $X$ is the quotient by $(\bG_m)^{\cD}$ of the union of non-degenerate 
orbits (an open subset of $\tX$, stable under $\tG$). In particular, the wonderful $G$-variety $X$ is uniquely 
determined by the monoid $\Gamma$.
 
As in \cite{BC08}, the first step in the proof is the determination of $T_{Z_0}(M_{\Gamma})$. Then a new 
ingredient is the vanishing of $T^2(X_0)^G$, an obstruction space for the functor of invariant infinitesimal 
deformations of $X_0$. This yields the smoothness of $\M_{\Gamma}$ at $Z_0$, which implies easily the
desired isomorphism $\M_{\Gamma} \cong \bA^r$.

\begin{example}\label{ex:won}
Let $G = \SL_2$ as in Example \ref{ex:sph}. Then the wonderful $G$-varieties $X$ are those of the following
list, up to $G$-isomorphism:

\medskip 

\noindent
(i) $\bP^1 = G/B$. Then $X$ is strict of rank $0$: it has no spherical root. Moreover, 
$R(X) = \Sym\big( V(1) \big)$, $\tG \cong \bG_m \times G$, and $\tX = V(1)$ where $\bG_m$ acts via scalar 
multiplication; the map (\ref{eqn:tcf}) is constant.

\medskip 

\noindent
(ii) $\bP^1 \times \bP^1$, of rank $1$ with open orbit $G/T$ and closed orbit $Y = G/B$ embedded as the
diagonal. Then $X$ is not strict, of rank $1$, and its spherical root is the simple root $\alpha$;
we have $Y = D^+ + D^-$ in $\Pic(X)$. 
Moreover, $R(X) \cong \Sym\big( V(1) \oplus V(1) \big)$, $\tG = G \times (\bG_m)^2$, and 
$\tX = V(1) \oplus V(1)$ where $(\bG_m)^2$ acts via componentwise multiplication, and $G$ acts diagonally.
The map (\ref{eqn:tcf}) is the determinant. The torus $C$ is one-dimensional, and the monoid $\Gamma$ 
has basis $(1,1)$ and $(1,-1)$.

\medskip 

\noindent
(iii) $\bP^2 = \bP\big( V(2) \big)$, of rank $1$ with open orbit $G/N_G(T)$ and closed orbit $Y = G/B$ embedded 
as the conic $(\Delta = 0)$. Here $X$ is strict, of rank $1$, and its spherical root is $2\alpha$; we have
$Y = 2D$ in $\Pic(X)$. Moreover, $R(X) = \Sym\big( V(2) \big)$,  $\tG \cong G \times \bG_m$, and $\tX = V(2)$ 
where $\bG_m$ acts via scalar multiplication. The map (\ref{eqn:tcf}) is the discriminant $\Delta$. The torus 
$C$ is trivial, and $\Gamma$ is generated by $2$. We have $\M_{\Gamma} = \Hilb^G_{h_2}\big( V(2) \big) \cong \bA^1$.
Note that the monoid generated by $4$ yields the same wonderful variety.  
\end{example}

\end{document}